\documentclass[11pt]{amsart}
\usepackage[active]{srcltx}
\usepackage{mathrsfs}
\usepackage[T1]{fontenc}

\usepackage{latexsym,enumerate}
\usepackage{amsmath,amssymb,amsthm,amsfonts,latexsym}
\usepackage[bookmarks]{hyperref}

\hypersetup{backref, colorlinks=true}
\hypersetup{backref, colorlinks=true, colorlinks   = true,
	urlcolor     = blue, linkcolor = blue, citecolor   = magenta
}

\def\adh#1{\overline{#1}}

\let\=\partial

\newtheorem {pro}{Proposition}[section]
\newtheorem {thm}[pro]{Theorem}
\newtheorem {cor}[pro]{Corollary}
\newtheorem{lem}[pro]{Lemma}

\theoremstyle{definition}
 \newtheorem {rem}[pro]{Remark}
\newtheorem {dfn}[pro]{Definition}

\newtheorem {exa}[pro]{Example}

\newtheorem {step}{Step}

\newcommand{\omt}{ \tilde{\Omega}}

\newcommand{\jac}{\mbox{jac}\,}

\newcommand{\tra}{\mathbf{tr}}

\newcommand{\R}{\mathbb{R}}

\newcommand{\N}{\mathbb{N}}
\newcommand{\cc}{\mathscr{C}}
\newcommand{\st}{\mathscr{S}}

\newcommand{\pms} {\mathbf{p_{_M}}}\newcommand{\pom} {\mathbf{p_{_\omd}}}

\newcommand{\et}{\quad \mbox{and} \quad }

\newcommand{\hn}{\mathcal{H}}

\newcommand{\oca}{\mathcal{O}}

\newcommand{\mba}{ {\overline{M}}}\newcommand{\omd}{{\Omega}}

\newcommand{\mep}{ {M^\ep}}
\newcommand{\uxo}{U_\xo}

\newcommand{\nep}{ {N^\ep}}

\newcommand{\alphat}{{\tilde{\alpha}}}
\newcommand{\mut}{\tilde{\mu}}
\newcommand{\pet}{\tilde{p}}
\newcommand{\tir}{\tilde{r}}

\newcommand{\F}{\mathcal{F}}
\newcommand{\ep}{\varepsilon}\newcommand{\epd}{\frac{\varepsilon}{2}}\newcommand{\epq}{\frac{\varepsilon}{4}}

\newcommand{\pa}{\partial}

\newcommand{\hh}{\mathcal{V}}

\newcommand{\bou}{\mathbf{B}}
\newcommand{\sph}{\mathbf{S}}
\newcommand{\orn}{{0_{\R^n}}}

\newcommand{\Bb}{\overline{ \mathbf{B}}}

\newcommand{\pau}{\pa}
\newcommand{\pad}{\pa}

\newcommand{\xo}{{x_0}}
\newcommand{\mc}{{\check{M}}}

\newcommand{\supp}{\mbox{\rm supp}}

\title[]{On Sobolev spaces of bounded subanalytic manifolds}

\makeatletter
 \thanks{Research partially supported
by the NCN grant  2021/43/B/ST1/02359.}
\@namedef{subjclassname@2020}{%
\textup{2020} Mathematics Subject Classification}
\makeatletter

\@addtoreset{equation}{section}

\makeatother

\author[ G. Valette]{ Guillaume Valette}

\address[G. Valette]{Instytut Matematyki Uniwersytetu
Jagiello\'nskiego, ul. S. \L ojasiewicza 6, Krak\'ow, Poland}\email{guillaume.valette@im.uj.edu.pl}



\begin{document}

\begin{abstract} We focus on the Sobolev spaces of bounded subanalytic submanifolds of $\R^n$. 
 We prove that if $M$ is such a manifold then the space $\cc_0^\infty(M)$ is dense in $W^{1,p}(M,\pa M)$ (the kernel of the trace operator) for all $p\le\pms$, where $\pms$ is the codimension in $M$ of the singular locus of $\mba\setminus M$ (which is always at least $2$). In the case where $M$ is normal, i.e. 
when  $\bou(x_0,\ep)\cap M$ is connected for every $x_0\in\mba$ and $\ep>0$ small, we show that $\cc^\infty(\mba)$
 is dense in $W^{1,p}(M)$ for all such $p$. This yields some duality results between $W^{1,p}(\omd,\pa \omd)$ and $W^{-1,p'}(\omd)$ in the case where $1< p\le \pom$ and $\omd$ is a bounded subanalytic open subset of $\R^n$.   As a byproduct, we deduce uniqueness of the (weak) solution of the Dirichlet problem associated with the Laplace equation. We then prove a version of Sobolev's Embedding Theorem for subanalytic bounded manifolds, show Gagliardo-Nirenberg's inequality (for all $p\in [1,\infty)$), and derive some versions of Poincar\'e-Friedrichs' inequality. We finish with a generalization of Morrey's Embedding Theorem.

\end{abstract}
\maketitle
\begin{section}{Introduction}
This article develops the theory of Sobolev spaces on subanalytic bounded submanifolds of $\R^n$, not necessarily compact. It was shown in \cite{poincwirt} that Poincar\'e-Wirtinger inequality holds on subanalytic domains. In \cite{trace}, we focused on the Sobolev space $W^{1,p}(M)$, $M$ bounded subanalytic manifold, in the case where $p$ is large, on which we  defined a trace operator and established that it is continuous and that the smooth compactly supported functions are dense in its kernel. The aim is to develop the basic material to study partial differential equations on bounded subanalytic open subsets of $\R^n$ and more generally on subanalytic manifolds.

 In the present article, we prove a version of Sobolev's Embedding Theorem for all $p\in [1,\infty)$ (Theorem \ref{thm_embedding}) for bounded subanalytic manifolds, asserting that there is a positive real number $\mu$ such that for all $p$,  the space $W^{1,p}(M)$ can be naturally compactly embedded in $L^{p+\mu}$, where $M$ is as above. We also generalize the famous closely related Gagliardo-Nirenberg inequality and provide some  Poincar\'e-Friedrichs type inequalities.  We also prove, in the case where $p$ is large,  that the elements of $W^{1,p}(M)$ are H\"older continuous with respect to the inner metric (Morrey's embedding, Corollary \ref{cor_morrey}).    The real number $\mu$ provided by our Sobolev Embedding Theorem is not only related to the dimension, as it is the case on Lipschitz domains, and actually heavily depends on the Lipschitz geometry of the singularities of the frontier of $M$. That it does not depend on $p$ is however valuable for applications and suffices to show for instance that any measurable function on such a manifold $M$ which is $L^1$ and such that all its subsequent partial derivatives are also $L^1$, must coincide almost everywhere with a function which is Lipschitz with respect to the inner metric (Corollary \ref{cor_infty}).  Such a result could be generalized to manifolds that are definable in a polynomially bounded o-minimal structure \cite{vdd, costeomin} but clearly fails as soon as the structure is not polynomially bounded (see Example \ref{exa_cor_infty}).

 Although subanalytic domains do not possess the cone property, it is well known since the original work of S. \L ojasiewicz \cite{lojdiv} that they are locally homeomorphic to cones, and it was established that this homeomorphism can be chosen subanalytic \cite{costeomin}. 
 We rely on the more recent results of \cite{gvpoincare,  livre}, where  a conic structure theorem for the Lipschitz geometry at a singular point of a subanalytic set (see Theorem \ref{thm_local_conic_structure}) was achieved. This theorem makes it possible to obtain
  some local estimates of the $L^p$ norm of functions that have an $L^p$ derivative (see section \ref{sect_estimates_loc}) which, together with some arguments of stratification theory, enables us to show that the smooth compactly supported functions are dense in the kernel of the trace operator in the case where $p$ is close to $1$ (Theorem \ref{thm_dense_lprime} and Corollary \ref{cor_densite_A_vide}). The provided real number underneath which this density result holds, denoted $\pms$, is the codimension in $M$ of the singularities of the frontier (which is always at least $2$). In particular, our density result holds in $W^{1,2}(M)$, which is the most important Sobolev  space.    Moreover, it entails some duality results between $W^{1,p}(\omd, \pa \omd)$ and $W^{-1,p'}(\omd)$ for $p\le \pom$, where $p'$ stands for the H\"older conjugate of $p$ and $\omd$ for a bounded open subanalytic subset of $\R^n$ (Proposition \ref{pro_dual}).

In \cite{bosmil}, L. P.  Bos and P. Milman prove some interesting Sobolev-Gagliardo-Nirenberg type inequalities for the functions of the space $\cc^\infty(\mba)$. Our purpose is however less to provide new inequalities in this framework than to illustrate that the recent developments of singularity theory make it possible to carry out a fully satisfying theory of the much wider class constituted by the Sobolev spaces of these manifolds, sufficiently rich to solve PDE on subanalytic domains. It was already observed \cite{adams} that Sobolev's and Morrey's embeddings theorems can be extended to the case where $M$ is an open subset of $\R^n$ that admits within its closure only isolated singularities that are metrically homogeneous (called domains with cusps in \cite{adams}). The  strength of our approach is however that we do not need any ad hoc assumption on the Lipschitz geometry of the considered bounded subanalytic manifold.

  As an application of our density theorem (Theorem \ref{thm_dense_lprime}), we also establish (Theorem \ref{thm_dir}) existence and uniqueness of the solution of the Dirichlet problem associated with the Laplace equation in $W^{1,2}(\omd)$, if $\Omega$ is an open bounded subanalytic set of any  dimension.  This result is closely related to the question  of the continuity of generalized Green functions to which T. Kaiser answered positively on subanalytic domains of $\R^2$ and $\R^3$ \cite{k1} (see also \cite{k2}).  The theory of Sobolev spaces that we develop in the present article provides a versatile framework to investigate this kind of problem.  It is indeed well-known that the generalized Dirichlet solution studied in \cite{k1} coincides with the classical one, when the latter solution exists. Noteworthy is the fact that our result brings new information on the class of integrability of the solution (and many improvements of this seem to be allowed by the theory) and makes it possible to compute effective approximations of solutions.
  
 It is worth stressing the fact that this theory actually goes over to the more general framework of polynomially bounded o-minimal structures (expanding $\mathbb{R}$) \cite{costeomin, vdd}.  Indeed, the arguments being essentially local, we only need the underlying manifold $M$ to be such that every $x\in \mba$ has a neighborhood $U$ in $\mba$ such that $U\cap M$ is bi-Lipschitz homeomorphic to a set definable in some polynomially bounded o-minimal structure (that may depend on $x$), and more generally we could adapt the statements to abstract varieties, which is useful for the complex case. Moreover, since we could use ``cut-off'' functions at infinity, the density results actually do not require the underlying manifold to be bounded.

\begin{subsection}{Some notations}Throughout this article,  $n$, $j$, and $k$ stand for  integers, the letter 
$M$  stands for a  bounded subanalytic $\cc^\infty$ submanifold of $\R^n$ and $m$ for its dimension, while the letter  $\Omega$ stands for a bounded definable open  subset of $\R^n$.

The origin of $\R^n$ will be denoted $0_{\R^n}$, although  the subscript $\R^n$ will be omitted when the ambient space is obvious from the context. We write $x\cdot y$ for the euclidean inner product of $x$ and $y$,  and $|.|$  for the euclidean norm.  Given $x\in  \R^n$ and $\ep>0$, we respectively denote by $\sph(x,\ep)$ and  $\bou(x,\ep)$ the sphere and the open ball of radius $\ep$ that are centered  at $x$ (for the euclidean norm), while $\adh{\bou}(x,\ep)$ will stand for the corresponding closed ball.

A mapping $h:A\to B$, $A\subset \R^n, B\subset \R^k$, is {\bf Lipschitz} if there is a constant $C$ such that $|h(x)-h(x')|\le C|x-x'|$ for all $x$ and $x'$. It is {\bf bi-Lipschitz} if it is a homeomorphism and if in addition $h$ and $h^{-1}$ are both Lipschitz.  Given two nonnegative functions $\xi$ and $\zeta$ on a set $E$ as well as a subset $Z$ of $E$, we write ``$\xi\lesssim \zeta$ on $Z$'' or  ``$\xi(x)\lesssim \zeta(x)$ for $x\in Z$'' when there is a constant $C$ such that $\xi(x) \le C\zeta(x)$ for all $x\in Z$.

 If $E\subset \R^n$ and $i\in \N\cup \{\infty\}$, we will write $\cc^i(E)$ for the space of those functions on $E$ that extend to a $\cc^i$ function on an open neighborhood of $E$ in $\R^n$.  We denote the closure of $E$ by $\adh{E}$ and set $\delta E:=\adh{E}\setminus E$.   We denote by $\hn^k$ the $k$-dimensional Hausdorff measure and by $L^p(E,\hn^k)$ the set of $L^p$ functions on $E$ with respect to this measure.

 By ``manifold'', we will always mean submanifold of $\R^n$, and a submanifold of $\R^n$ will always be endowed with its canonical measure, provided by volume forms.  As integrals will always be considered with respect to this measure, we will not indicate the measure when integrating on a submanifold of $\R^n$.

 Given a measurable mapping $v:M\to \R^k$ (with respect to the canonical measure of $M$) and an open subset $U$ of $M$, for each $p\in [1,\infty)$ we denote by $||v||_{L^p(U)}$ the (possibly infinite) $L^p$ norm  of the restriction of $v$ to $U$. As usual, we denote by $L^p(M)$ the set of measurable functions $u$ on $M$ for which  $||u||_{L^p(M)}$ is finite, and set $$W^{1,p}(M):= \{u\in L^p(M),\; |\partial u| \in L^p(M)\},$$ where $\pa u$ stands for the gradient of $u$ in the sense of distributions.
 As well-known,  this space, equipped with the norm
$$||u||_{W^{1,p}(M)}:=||u||_{L^p(M)}+| |\partial u||_{L^p(M)},$$ is a Banach space, in which
  $\cc^\infty(M)$ is dense  for all $p\in [1,\infty)$. 
  
  We denote by $\supp\, u$ the support of a distribution $u$ on $M$, and, given an open subset $ Z$  of $\mba$ containing $M$, we write $\supp_Z u$ for  the {\bf support of $u$ in $Z$}, defined as the closure in $Z$ of $\supp\, u$.

  We will write $\cc_0^\infty(M)$ for the space of elements of $\cc^\infty(M)$ that are compactly supported, while $W^{1,p}_0(M)$ will stand for the closure of this space in $W^{1,p}(M)$. We will regard $\cc^\infty(\mba)$ as a subset of $W^{1,p}(M)$.  In particular, for $u\in \cc^\infty(\mba)$,  $\supp\, u$ will be a subset of $M$ and
   we set for $Z$ open in $\mba$ $$\cc^\infty_{Z}(\mba):=\{u\in \cc^\infty(\overline{M}):\overline{ \supp\, u} \subset Z  \}.$$
Like  $\cc^\infty(\mba)$, the space $\cc^\infty_{Z}(\mba)$ will be regarded as a subspace of $\cc^\infty(M)$.
Note that when $Z\supset M$, the space $\cc^\infty_{Z}(\mba)$ is nothing but the set constituted by the functions $u\in \cc^\infty(\mba)$ that are {\bf compactly supported in $Z$}, i.e., for which $\supp_Z u$ is compact. 

Given two functions $u$ and $v$ on  $\omd$ such that $uv$ is $L^1$, we set $<u,v>:=\int_\omd uv.$  
  We also denote by $\pa_i u$ the $i^{th}$ partial derivative in the sense of distributions and,  for simplicity of statements, let $\pa_0 u:=u$.
 
 Given a mapping $h:M\to M$, we write $\jac\, h$  for the absolute value of the determinant of $Dh$, the derivative of $h$ (defined at the points where $h$ is differentiable). 
 We write $p'$ for the H\"older conjugate of $p$, i.e., $p'=\frac{p}{p-1}$, and, given a topological vector space $V$, we denote by $V'$ its topological dual.

  We shall make use of the following form of H\"older's inequality. Let $p_1,\cdots,p_\kappa$ and $\theta_1,\dots, \theta_\kappa$ be positive real numbers and let $q\in [1,\infty)$ be such that $\frac{1}{q}=\sum_{i=1} ^\kappa \frac{\theta_i}{p_i}$.  If  $g_1,\dots,g_\kappa$ are measurable functions such that $g_i\in L^{p_i}(M)$, for all $i$, then we have 
 \begin{equation}\label{eq_holder}
  ||g_1^{\theta_1}\cdots g_\kappa ^{\theta_\kappa}||_{L^q(M)}\le ||g_1 ||^{\theta_1} _{L^{p_1}(M)}\cdots ||g_\kappa|| _{L^{p_\kappa}(M)}^{\theta_\kappa}.
 \end{equation}
 \end{subsection}

\end{section}


\begin{section}{Subanalytic sets}
 We refer the reader to  \cite{bm, ds, l, livre} for all the basic facts about subanalytic geometry.  Actually, \cite{livre} also gives a detailed presentation of the Lipschitz properties of these sets, so that the reader can find there all the needed facts to understand the results achieved in the present article.

\begin{dfn}\label{dfn_semianalytic}
A subset $E\subset \R^n$ is called {\bf semi-analytic} if it is {\it locally}
defined by finitely many real analytic equalities and inequalities. Namely, for each $a \in   \R^n$, there is
a neighborhood $U$ of $a$ in $\R^n$, and real analytic  functions $f_{ij}, g_{ij}$ on $U$, where $i = 1, \dots, r, j = 1, \dots , s_i$, such that
\begin{equation}\label{eq_definition_semi}
E \cap   U = \bigcup _{i=1}^r\bigcap _{j=1} ^{s_i} \{x \in U : g_{ij}(x) > 0 \mbox{ and } f_{ij}(x) = 0\}.
\end{equation}

The flaw of the  semi-analytic category is that  it is not preserved by analytic morphisms, even when they are proper. To overcome this problem, we prefer working with the  subanalytic sets, which are defined as the projections of the semi-analytic sets.

A subset $E\subset \R^n$  is  {\bf  subanalytic} if 
 each point $x\in\R^n$ has a neighborhood $U$ such that $U\cap E$ is the image under the canonical projection $\pi:\R^n\times\R^k\to\R^n$ of some relatively compact semi-analytic subset of $\R^n\times\R^k$ (where $k$ depends on $x$).
   
   A subset $Z$ of $\R^n$ is  {\bf globally subanalytic} if $\hh_n(Z)$ is a subanalytic subset of $\R^n$, where $\hh_n : \R^n  \to (-1,1) ^n$ is the homeomorphism defined by $$\hh_n(x_1, \dots, x_n) :=  (\frac{x_1}{\sqrt{1+|x|^2}},\dots, \frac{x_n}{\sqrt{1+|x|^2}} ).$$

   We say that {\bf a mapping $f:A \to B$ is  subanalytic} (resp. globally subanalytic), $A \subset \R^n$, $B\subset \R^m$ subanalytic (resp. globally subanalytic), if its graph is a  subanalytic  (resp. globally subanalytic) subset of $\R^{n+m}$. In the case $B=\R$, we say that  $f$ is a (resp. globally) {\bf  subanalytic function}. For simplicity globally subanalytic sets and mappings will be referred as {\bf definable} sets and mappings (this terminology is often used by o-minimal geometers \cite{vdd,costeomin}).

   \end{dfn}

The globally subanalytic category is very well adapted to our purpose.  It is stable under intersection, union, complement, and projection. It thus constitutes an o-minimal structure \cite{vdd, costeomin}, and consequently admits cell decompositions and stratifications, from which it comes down that definable sets enjoy a large number of finiteness properties (see \cite{costeomin,livre} for more).

\subsection*{\L ojasiewicz's inequality.} Some results about $L^p$ functions will be valid {\it for $p$ sufficiently large}, which means that there will be $p_0\in \R$ such that the claimed fact will be true for all $p\in (p_0,\infty)$ (i.e. when we do not specify that $p$ may be infinite, it should be understood that it is finite). 
This number $p_0$ will be provided by \L ojasiewicz's inequality, which is one of the main tools of subanalytic geometry. It originates in the work of S. \L ojasiewicz \cite{lojdiv}, who established this inequality in order to answer a problem  about distribution theory.  We shall make use of the following version (see the proof of (\ref{ln})):


 \begin{pro}\label{pro_lojasiewicz_inequality}(\L ojasiewicz's inequality)
Let $f$ and $g$ be two globally subanalytic functions on a  globally subanalytic set $A$ with $\sup\limits_{x\in A} |f(x)|<\infty$. Assume that
 $\lim\limits_{t \to 0} f(\gamma(t))=0$ for every
globally subanalytic arc $\gamma:(0,\ep) \to A$ satisfying $\lim\limits_{t \to 0} g(\gamma(t))=0$.

Then there exist $\nu \in \N$ and $C \in \R$ such that for any $x \in A$:
$$|f(x)|^\nu \leq C|g(x)|.$$
\end{pro}

See  for instance \cite{li} for a proof. 

\subsection*{The set $\pa M$.} 
We will say that $M$ is {\bf Lipschitz regular} at $x\in \delta M$ if this point has a neighborhood $U$ in $\R^n$ such that each connected component of $U\cap M$ is the interior of a Lipschitz manifold with boundary $U\cap \delta M$. 
 We then let:
$$\pau M:=\{x\in \delta M: M\mbox{ is Lipschitz regular at $x$}\}.$$
This set is definable and we have
\begin{equation}\label{eq_pa1}\dim \left(\delta M \setminus  \pau M\right)\le m-2.\end{equation}
 This fact follows from a famous result  which is sometimes referred as {\it Wing Lemma} by geometers (see \cite[Proposition 1, section 19]{l}, \cite[Proposition 9.6.13]{bcr}, or \cite[Lemma 5.6.7]{livre}).   The closure of the set $\pa M$ in $\R^n$ is called the {\bf fence} of $M$.

 \begin{exa} If $\omd$ is a definable open subset of $\R^n$  then $\adh{\omd}$ is the union of (one or two)  $\cc^\infty$ submanifolds   with boundary of $\R^n$ at $\hn^{n-1}$-almost every point of $\delta \omd$.
It is worth here underlining that, when $M$ has positive codimension in $\R^n$, it is not true that $\mba$ is a union of $\cc^\infty$  manifolds with boundary at $\hn^{m-1}$-almost every point of $\delta M$. A counterexample is given by the graph of the function $f(x)=x^{3/2}$, $x\in [0,1]$,  which is not a $\cc^2$ manifold with boundary at the origin. It is however true that $\mba$ is a finite union of $\cc^1$ manifolds with boundary at $\hn^{m-1}$-almost every point of $\delta M$.\end{exa}

\begin{subsection}{Stratifications.}  It is a very natural idea to divide a set which is singular into $\cc^\infty$ manifolds. This is the purpose of the definition below which, together with Theorem \ref{pro_existence_stratifications},  will be useful to work locally at a point of the frontier of the underlying manifold. 

\begin{dfn}\label{dfn_stratifications}
 A {\bf
stratification} of a subset of $ \R^n$ is a finite partition of it into
definable $\cc^\infty$ submanifolds of $\R^n$, called {\bf strata}. A stratification is {\bf compatible} with a set if this set is the union of some strata. A {\bf refinement} of a stratification $\Sigma$ is a stratification $\Sigma'$ compatible with the strata of $\Sigma$.


 A  stratification $\Sigma$ of a set $X$ is {\bf locally bi-Lipschitz trivial} if for every $S\in \Sigma$, there are an open neighborhood $V_S$ of $S$ in $X$ and a smooth  retraction $\pi_S:V_S\to S$  such that every $x_0\in S$ has an open neighborhood $W$ in $S$ for which there is a  bi-Lipschitz homeomorphism  $$\Lambda:\pi_S^{-1}(W)\to \pi_S^{-1}(x_0) \times W, $$ satisfying:
  \begin{enumerate}[(i)]
   \item  $\pi_S(\Lambda^{-1}(x,y))= y$, for all $(x,y)\in \pi_S^{-1}(x_0)\times  W$.
   \item   $\Sigma_{x_0}:=\{  \pi_S^{-1}(x_0)\cap Y:Y\in \Sigma\} $ is a stratification of $ \pi_S^{-1}(x_0)$,  and $\Lambda(\pi_S^{-1}(W)\cap Y)=(\pi_S^{-1}(x_0)\cap Y)\times W$, for all $Y\in \Sigma$.
  \end{enumerate}
\end{dfn}

T. Mostowski \cite[Proposition $1.2$]{m} proved that every complex analytic set admits a locally bi-Lipschitz trivial stratification. This result was extended to the subanalytic category by A. Parusi\'nski  \cite{stratlip}, to  polynomially bounded o-minimal structures expanding $\R$ in \cite{lipsomin}, and to polynomially bounded o-minimal structures expanding an arbitrary real closed field in \cite{halyin}.  The author gives in  \cite{gvhandbook} a construction of  locally  bi-Lipschitz trivial stratifications for definable sets in polynomially bounded o-minimal structures (expanding $\R$) for which the local trivializations are in addition definable. The stratifications constructed in \cite{stratlip, lipsomin, gvhandbook} can be required to be compatible with finitely many definable sets, and therefore, to refine a given stratification.  Hence, the theorem below follows from  \cite[Theorem 1.4]{stratlip}, \cite[Theorem $2.6$]{lipsomin}, or \cite[Corollary $1.6.8$]{gvhandbook}.
\begin{thm}\label{pro_existence_stratifications}
  Every stratification $\Sigma$ can be refined into a stratification which is locally bi-Lipschitz trivial.
 \end{thm}

 \begin{rem}\label{rem_pi_S} \begin{enumerate}
    \item\label{item_whitney}    Let  $\Sigma$ be a stratification of $\mba$. We can require the   refinement of $\Sigma$  provided by the above theorem, that we will denote by $\Sigma'$, to be a refinement of a Whitney $(a)$ stratification of $\mba$ compatible with $M$ (see \cite[Ch. $2$, section $2.6$]{livre} for the definition and construction of Whitney $(a)$ regular refinements).  This ensures that for every $ S\in \Sigma'$ and $\pi:U\to S$  smooth retraction,
  $\pi^{-1}(x)\cap M$ is a smooth manifold for all $x\in S$, after shrinking $U$ to a smaller neighborhood of $S$ if necessary.    It is also worthy of notice that we can require that the only stratum of maximal dimension is $M$, taking together all the strata included in $M$ (as the local trivializations preserve these strata, they preserve $M$).

\item\label{item_retractions} Mostowski's Lipschitz stratifications admit a trivialization of any given smooth retraction $\pi_S:V_S\to S$, if $S$ is a stratum. Corollary $1.6.8$ of \cite{gvhandbook} provides a stratification such that a given smooth definable retraction, that can be for instance the closest point retraction (which is definable \cite[Proposition $2.4.1$]{livre}), can be trivialized. Hence, the stratification as provided by the above theorem admits  {\it definable} smooth retractions on strata that are bi-Lipschitz trivial.
                        \end{enumerate}
 \end{rem}

\end{subsection}
\begin{subsection}{Lipschitz conic structure}\label{sect_lcs}
 The main ingredient of our approach is the following result achieved  in \cite[Theorem $3.1$]{gvpoincare} (see also Theorem $3.4.1$ of the survey \cite{livre} which gives in addition a complete expository of the theory). In this theorem, $x_0* (\sph(x_0,\ep)\cap X)$ stands for the cone over $\sph(x_0,\ep)\cap X$ with vertex at $x_0$.

\begin{thm}[Lipschitz Conic Structure]\label{thm_local_conic_structure}
  Let  $X\subset \R^n$ be subanalytic and $x_0\in X $. 
For $\ep>0$ small enough, there exists a Lipschitz subanalytic homeomorphism
$$H: x_0* (\sph(x_0,\ep)\cap X)\to  \Bb(x_0,\ep) \cap X,$$  
  satisfying $H_{| \sph(x_0,\ep)\cap X}=Id$, preserving the distance to $x_0$, and having the following metric properties:
\begin{enumerate}[(i)] 
 \item\label{item_H_bi}     The natural retraction by deformation onto $x_0$ $$r:[0,1]\times  \Bb(x_0,\ep)\cap X \to \Bb(x_0,\ep)\cap X,$$ defined by $$r(s,x):=H(sH^{-1}(x)+(1-s)x_0),$$ is Lipschitz.   
 Indeed, there is a constant $C$ such that  for every fixed $s\in [0,1]$, the mapping $r_s$ defined by $x\mapsto r_s(x):=r(s,x)$, is $Cs$-Lipschitz.
 \item \label{item_r_bi}  For each $\delta>0$,
 the restriction of $H^{-1}$ to $\{x\in X:\delta \le |x-x_0|\le \ep\}$ is Lipschitz and, for each $s\in (0,1]$, the map  $r_s^{-1}:\Bb(x_0,s\ep) \cap X\to \Bb(x_0,\ep) \cap X$ is Lipschitz. 
\end{enumerate}
\end{thm}
\begin{exa}\label{exa_lcs} The mapping $H$ of the above theorem cannot be required to be bi-Lipschitz, as shown by the example
 $$X:=\{(x,y)\in [0,1]\times \R:|y|\le x^2\}.$$ 
 An explicit formula for the corresponding mapping $H$ at $x_0=(0,0)$ that can help the reader to grasp the properties listed in the above theorem is provided in \cite[Example $1.6$]{poincwirt}.  It is also relevant to mention an example which shows that this theorem fails on non polynomially bounded o-minimal structures. Let
 $$Y:=\{(x,y)\in (0,1]\times \R: |y|\le e^\frac{-1}{x} \}\cup \{(0,0)\}. $$
If there were such a mapping $H$, it is not difficult to see that the points $z:=(x,e^\frac{-1}{x})$ and $z':=(x,-e^\frac{-1}{x})$, for $x\in (0,1]$, would satisfy for $s\in (0,1]$:
$$|r_s(z)-r_s(z')|\le C e^\frac{-1}{sx}, $$
for some constant $C$ independent of $x$ and $s$.
As $|z-z'|=2e^\frac{-1}{x}$, we see that  $r_s^{-1}$ would fail to be Lipschitz for each $s\in (0,1)$.  This example enlightens the importance of polynomial boundedness and the key role played by \L ojasiewicz's inequality in the theory (see also Example \ref{exa_cor_infty}).
 \end{exa}

\begin{rem}\label{rem_lcs2}
	The Lipschitz constant of $r_s^{-1}$ (see $(ii)$) is bounded away from infinity if $s$ stays bounded away from $0$. Indeed, if $s\ge \delta >0$ then $r_\delta =r_{\frac{\delta}{s}}\circ r_s$ entails $r_s^{-1}=r_{\frac{\delta}{s}}\circ r_\delta^{-1}$,  and the  Lipschitz constants of both $r_{\frac{\delta}{s}}$ and $r_\delta^{-1}$ are bounded independently of $s\ge \delta$.
\end{rem}

\subsection{\bf Lipschitz conic structure at a point of $\delta M$.}\label{sect_lcs_mba}
We now wish to explain in which setting we will apply the above theorem and then mention a few facts that were established in \cite{trace}. 
Let us assume that $0_{\R^n}\in\delta M$ and 
set for $\eta>0$: \begin{equation}\label{eq_metaneta}M^{\eta}:=\bou(0_{\R^n},\eta)\cap M\;\; \et\;\; N^{\eta}:=\sph(0_{\R^n},\eta)\cap M.\end{equation}

  Apply Theorem \ref{thm_local_conic_structure} with $X=M\cup\{0_{\R^n}\}$ and $x_0=0_{\R^n}$. Fix $\ep>0$ sufficiently small for the statement of the theorem to hold and for $N^\ep$ to be a $\cc^\infty$ manifold, and  let $r$ denote the mapping provided by this theorem. We will assume for convenience $\ep<\frac{1}{2}$. Observe  that since $r$ is subanalytic, it is $\cc^\infty$ almost everywhere.

  Since $r_s$ is bi-Lipschitz for every $s>0$, $\jac\, r_s$ can only tend to zero if $s$ is itself going to zero (see Remark \ref{rem_lcs2}). Hence, by \L ojasiewicz's inequality (see Proposition \ref{pro_lojasiewicz_inequality}),
there are a positive integer $\nu$ and a constant $\kappa$ such that for all $s\in (0,1)$ we have for almost all $x\in M^\ep$:
\begin{equation}\label{eq_jacr_s}
 \jac \,r_s(x) \ge \frac{s^{\nu}}{\kappa}.  
\end{equation}

\noindent{\bf A few facts.} The following facts were then established in \cite[section $2.1$]{trace}. There is a constant $C$ such that:

\begin{enumerate}
\item For all $s\in (0,1)$ we have for almost all $x\in M^\ep$:  \begin{equation}\label{eq_der_r_s}
       \left|\frac{\pa r}{\pa s}(s,x)\right|\le C|x|.
      \end{equation}
       \item For each $v\in L^{p}(M^\ep)$,  $p\in [1,\infty)$, we have for all $\eta\in (0,\ep]$: \begin{equation}\label{eq_coarea_sph}
  \frac{1}{C}  \left(\int_0 ^\eta ||v||_{L^p (N^\zeta)}^p d\zeta \right)^{1/p}  \le    ||v||_{L^p (M^\eta)} \leq C \left(\int_0 ^\eta ||v||_{L^p (N^\zeta)}^p d\zeta \right)^{1/p}.
       \end{equation}
       \item There exists $\nu\in \N$ such that for each $v\in L^{p}(M^\ep)$,  $p\in [1,\infty)$,  $\eta\in(0,\ep]$, and $s\in (0,1)$:
\begin{equation}\label{ln}
||v\circ r_s||_{L^p(N^\eta)}\le Cs^{-\nu/p}||v||_{L^p(N^{s\eta})}   
\end{equation}
and 
\begin{equation}\label{lnm}
||v\circ r_s||_{L^p(M^\eta)}\le C s^{-\nu/p}||v||_{L^p(M^{s\eta})}  . 
\end{equation} 
\end{enumerate}
 
\end{subsection}

\end{section}


\begin{section}{Trace operators on subanalytic manifolds}\label{sect_prel} 
We are going to define a trace operator $\tra_{\pau M} : W^{1,p}(M)\to L^p_{loc}(\pad M)$ and show that it is continuous (in section \ref{sect_traceloc}), which requires to recall (in section \ref{sect_normal}) a few results about normal spaces established in \cite{trace}.  We will also derive from the results of \cite{trace} that, when $p$ is large, $\tra_{\pau M} u$ is actually $L^p$ for all $u\in W^{1,p}(M)$ (in section \ref{sect_trace_p_large}).

  \begin{subsection}{Normal manifolds.}\label{sect_normal}The notions of normal manifold and $\cc^\infty$ normalization, which originate in \cite{trace}, are natural generalizations to our framework of the notions of normal pseudomanifold  and normalization introduced in \cite{mcc,ih1}.

  \begin{dfn}\label{dfn_embedded}
We say that $M$ is {\bf connected at $x\in \delta M$}\index{connected at $x$} if  $\bou(x,\ep)\cap M$ is connected for all  $\ep>0$ small enough.
We say that $M$ is {\bf normal} if it is connected at each $x\in \delta M$.

We will say that it is {\bf connected along $Z\subset \delta M$} if it is connected at each point of $Z$.
The manifold $M$ will be said to be {\bf weakly normal} if there is a set $E\subset \delta M$ such that $\dim (\delta M\setminus E)\le m-2$, and along which $M$ is connected, 
which amounts to require that $M\cup \pau M$ is a Lipschitz manifold with boundary $\pa M$.  Of course, every normal manifold is  weakly normal. 

{\bf A $\cc^\infty$ normalization of $M$ }
 is  a definable $\cc^\infty$ diffeomorphism $h: \mc\to M$ satisfying $\sup_{x\in \mc} |D_x h|<\infty$ and  $\sup_{x\in M} |D_x h^{-1}|<\infty$,  with $\mc$ normal $\cc^\infty$ submanifold of $\R^k$, for some $k$. We will sometimes say that $\mc$ is a normalization of $M$.
 \end{dfn}

  It can be seen \cite{trace} that if $h:\mc \to M$ is a normalization and  $x_0\in \delta M$ then the number of points of $h^{-1}(x_0)$ coincides with the number of connected components of  $\bou(x_0,\ep)\cap M$, $\ep>0$ small.
 The following proposition gathers Propositions $3.3$ and $3.4$ of \cite{trace}, yielding existence and uniqueness of $\cc^\infty$ normalizations.

 \begin{pro}\label{pro_normal_existence}\begin{enumerate}
                                         \item 
        Every bounded definable $\cc^\infty$ manifold admits a $\cc^\infty$ normalization.
\item  \label{item_unique}   Every $\cc^\infty$ normalization  $h: \mc\to M$  extends continuously to a mapping from $\adh{\mc}$ to $\mba$.
    Moreover, if $h_1:\mc_1 \to M$ and $h_2:\mc_2 \to M$ are two $\cc^\infty$ normalizations of $M$, then $h_2^{-1} h_1$  extends to a homeomorphism between $\adh{\mc_1}$ and $\adh{\mc_2}$.                             \end{enumerate}

    \end{pro}

\end{subsection}

\begin{subsection}{The trace on $\pau M$.}\label{sect_traceloc} In \cite{trace} we have defined some trace operators on $W^{1,p}(M)$ in the case where $p$ is large (this is recalled in the next section). We explain in this section that we can always (i.e. for any $p$) define a trace which is {\it locally} $L^p$ in $\pau M$. As $\mba$ is a finite union of Lipschitz manifolds with boundary at every point of $\pau M$, this actually follows from the classical theory. A similar trace operator was defined in \cite{poincfried} on open subsets of $\R^n$ that are definable in an o-minimal structure.

For  $p\in [1,\infty)$,
we say that a sequence of measurable functions $u_i\in W^{1,p}(M)$ converges to $u\in W^{1,p}(M)$ {\bf in the strong $W^{1,p}_{loc}$ topology} if every $x\in M\cup \pa M$ has a neighborhood $U$ in $\R^n$ such that $u_{i|U\cap M}$ and  $\pa u_{i|U\cap M}$ respectively tend to $u_{|U\cap M}$ and $\pa u_{|U\cap M}$ in the $L^p$ norm (we say ``strong'' because it is local in $M\cup \pa M$, not only in $M$).

\subsection*{The case of a normal manifold.} We first define our trace in the case where $M$ is normal (weakly normal is actually enough for what we do in the proof below). In this case, $M\cup \pau M$ is a Lipschitz manifold with boundary.

\begin{pro}\label{pro_trace_loc} If $M$ is normal then for any  $p\in [1,\infty)$, 
 the mapping
 \begin{equation}\label{eq_cinfty_traceloc}\cc^\infty(\mba)\to  L^p_{loc}(\pau M), \quad \varphi\mapsto \varphi_{|\pau M}\,,\end{equation} 
 extends to a mapping
 $$\tra_{\pau M} :W^{1,p} (M)\to L^p_{loc}(\pau M)$$ 
 which is continuous in the strong $W^{1,p}_{loc}$ topology. 
\end{pro}
\begin{proof}
 The classical theory on Lipschitz manifolds with boundary implies that for every compact subset $K$ of $\pau M$ and every open subset $V$ of $M\cup \pau M$ containing $K$, there is a constant $C$ (depending on $K$ and $V$) such that for all $\varphi \in \cc^\infty(M\cup \pau M)$ we have $$ ||\varphi||_{L^p(K,\hn^{m-1})} \le C||\varphi||_{W^{1,p}(V)}.$$
 As   $M\cup \pau M$ is a manifold with boundary, the space $\cc^\infty(M\cup \pau M)$ is dense in $W^{1,p}(M)$  in the strong $W^{1,p}_{loc}$ topology,  and the mapping defined in (\ref{eq_cinfty_traceloc}) extends to a continuous mapping on $W^{1,p} (M)$. 
\end{proof}

 In the case of non necessarily normal manifolds, we are going to define a trace operator which is ``multi-valued'' (since $M$ may have several connected components locally at a generic frontier point).  

\subsection*{General definition of the trace.}
There are definable subsets $S_1,\dots, S_k$ of $\pau M$ satisfying  $\hn^{m-1}(\pau M \setminus\bigcup_{i=1}^k S_{i})=0$ and such that each $S_i$ has a neighborhood $U_i$ in $\R^n$ for which each connected component of $U_i\cap M$ is the interior of a manifold with boundary $S_i$ (such a covering is induced for instance by the images of the simplices of a triangulation of $M$ compatible with $\pa M$).  We denote by $Z_{i,1}, \dots, Z_{i,l_i}$  these connected components.

We first define the trace on $S_i$, denoted $\tra_{S_i}$, $i\in \{1,\dots,k\}$.
 Let $\tra^j_{i}:W^{1,p}(Z_{i,j})\to L^p_{loc}(S_i,\hn^{m-1})$ denote the operator obtained by applying Proposition \ref{pro_trace_loc} to $Z_{i,j}$  (which is normal if $U_i$ is sufficiently small), and then define  $$\tra_{S_i} : W^{1,p}(M) \to L^p_{loc} (S_i,\hn^{m-1})^l,\qquad  l=\underset{i=1,\dots,k}\max l_i\, ,$$  by setting for $u\in W^{1,p}(M)$ and $x\in S_i$: 
\begin{equation}\label{eq_tral}\tra_{S_i} u(x) := ( \tra_{i}^1 u_{|Z_{i,1}}(x), \dots , \tra_{i}^{l_i} u_{|Z_{i,l_i}}(x), 0,\dots, 0).\end{equation}
    Here, we add $(l-l_i)$ times the zero function because  it will be convenient that the trace has the same number of components for all $i$.   Note that, as $M\cup \pa M$ is a finite union of manifolds with boundary at every point of $\pa M$, $\tra_{S_i} u(x)$ is actually $L^p$ on a neighborhood of every point of $\adh{S_i}\cap\pa M $.
    
    This mapping of course depends on the way the elements $Z_{i,1},\dots, Z_{i,l_i}$ are enumerated. However,
up to permutation of the $l_i$ first components,  $ \tra_{S_i} u(x)$ is unique, and it only depends on the germ of $u$ near $S_i$. In particular, the kernel of this mapping is independent of any choice. 

We now can define $\tra_{\pau M}u$ as the function which coincides with $\tra_{S_i} u$ on $S_i\subset \pau M$, for each $i$. Since $\pau M\setminus \bigcup_{i=1}^ k S_i$ is $\hn^{m-1}$-negligible, this defines a function (almost everywhere) on $\pau M$. Moreover, as  the germ of $\tra_{i}^j u_{|Z_{i,j}}$ at a point of $\adh{S_i}\cap \pau M$ is $L^p$ for all $i$ and $j$, it is clear that this defines an element of $L^p_{loc}(\pau M)^l$. The above proposition yields that the mapping  $$\tra_{\pau M}:W^{1,p}(M)\to L^p_{loc}(\pau M)^l$$ is continuous  in the strong $W^{1,p}_{loc}$ topology.  

  For $p\in [1,\infty)$, we then set for an arbitrary open subset $Z$ of $\pau M$:\begin{equation*}\label{eq_wpa}
                              W^{1,p} (M,Z) :=\{ u\in W^{1,p}(M):\tra_{\pau M} u\equiv 0 \mbox{ on } Z\}.
                              \end{equation*}      
                              
                     We will see that $W^{1,p}(M,\pau M)$ and $W^{1,p}_0(M)$ coincide for small values of $p$ (Corollary \ref{cor_densite_A_vide}). We deal with this issue in the case ``$p$ large'' in the next section.
\begin{rem}
 It directly follows from our local definition of the trace that  a function $u\in W^{1,p}(\omd)$ belongs to $W^{1,p} (\omd,\pa \omd) $ if and only if it satisfies for all $\varphi\in \cc^\infty_{\omd\cup\pad \omd}(\adh{\omd})$ and all $i=1,\dots, n$: 
        \begin{equation}\label{eq_trace_nulle} <u,\pa_i \varphi>=-<\pa_i u,\varphi>. \end{equation}
\end{rem}
                                                                 
\begin{exa}\label{exa_trace}\cite{poincfried}
It is easy to produce examples of definable domains $\Omega$ admitting a function $u\in W^{1,p}(\omd)$ such that $\tra_{\pa \omd}u$ is not $L^p$ on the boundary. Let for $k>2$, 
 $$\Omega_k:=\{(x,y)\in (0,1)^2:y < x^k\},$$ and let $u(x,y):=\frac{1}{x}.$ 
 Clearly,  
 $u\in W^{1,p}(\omd_k)$, for each $p\in [1,\frac{k}{2}]$, while $\tra_{\pa \Omega_k} u$ is not $L^p$ on $\pa \omd_k$.
\end{exa}          
                              
                                        \end{subsection}


\begin{subsection}{The trace when $p$ is large.}\label{sect_trace_p_large} We now wish to explain that in the case where $p$ is large, the results of \cite{trace} yield that the mapping $\tra_{\pau M}$, defined in the preceding section, is continuous in the $L^p$ norm (see Corollary \ref{cor_trace_bounded} below and Example \ref{exa_trace} just above).

\begin{thm}\label{thm_trace}\cite{trace}
Assume that $M$ is normal and let $A$ be any definable subset of $\delta M$. For all $p\in [1,\infty)$ sufficiently large, we have:
\begin{enumerate}[(i)]
\item
$\cc^\infty(\adh{M})$ is dense in $W^{1,p}(M)$.
\item The linear operator
\begin{equation*}\label{trace}
\cc^\infty(\adh{M})\ni \varphi \mapsto \varphi_{|A}\in L^p(A,\hn^k), \qquad k:=\dim A,
\end{equation*}
is continuous in the norm $||\cdot ||_{W^{1,p}(M)}$ and thus extends to a mapping $\tra_A:W^{1,p}(M)\to L^p(A,\hn^k)$.
\item If  $\st$ is a stratification of $A$, then $\cc^\infty_{\adh{M}\setminus\adh{A}}(\mba)$ is a dense subspace of
$\bigcap\limits_{Y\in\st}\ker \tra_Y$.
\end{enumerate}\end{thm}
Using a normalization,  it is then not difficult to extend the notion of trace to the case of a not necessarily normal manifold (see \cite{trace} for details). Like in the previous section, this trace operator takes values in  $L^p(A,\hn^k)^l,$ for some $l\ge 1$.
We then have \cite[Corollary 3.7]{trace}:

\begin{cor}\label{cor_trace_bounded}
 Let $A\subset \delta M$ be a subanalytic set of dimension $k$.  For $p\in [1,\infty)$ sufficiently large, the linear operator $$\tra_A :W^{1,p}(M) \to L^p(A,\hn^k)^l,$$
  is continuous.
\end{cor}

Of course, since the trace is induced by restriction of functions to $\pau M$, in the case $A=\pau M$, this operator coincides with the trace operator defined in the previous section.
 We thus get that the mapping $\tra_{\pau M}:W^{1,p}(M)\to L^p(\pau M)^l$ is continuous for large values of $p\in [1,\infty)$. The situation is however here a bit more complicated than on Lipschitz manifolds with boundary, as it is not true that $\cc_0^\infty(M)$ is dense in $\ker \tra_{\pau M}$, i.e., as $W_0^{1,p}(M)$ can be strictly smaller than $W^{1,p}(M,\pa M)$:
\begin{exa}\label{exa_disk}
 If $M$ is the open unit disk in $\R^2$ centered at $0$ from which was taken off the center then every element of $\cc_0^\infty(M)$ must vanish in the vicinity of the origin. For every sufficiently large value of $p$, $\tra_{\{0\}} :W^{1,p}(M)\to L^p(\{0\})$ is well-defined (by Corollary \ref{cor_trace_bounded}), and the closure of $\cc_0^\infty(M)$ is included in the kernel of this linear mapping. Therefore, a smooth function $u$ which is $1$ near the origin and that has support in the disk of radius $\frac{1}{2}$ belongs to $W^{1,p}(M,\pa M)$ and not to $W_0 ^{1,p}(M)$.
\end{exa}
We however have \cite[Corollary $3.9$]{trace}:
\begin{cor}\label{cor_densite_nonnormal} If $\Sigma$ is a stratification of a subanalytic subset $A$ of $\delta M$ such that $h^{-1}(A)$ is dense in $\delta \mc$, for some (and hence for any) normalization $h:\mc\to M$, then
 $ \cc^\infty _0(M)$ is dense in $
\underset{S\in \Sigma}{\bigcap}\ker \tra_S $ for all $p\in [1,\infty)$ sufficiently large.
\end{cor}

When $M$ is an open subset of $\R^n$, it is easy to see that the above condition ``$h^{-1}(A)$ is dense in $\delta \mc$''  merely amounts to ``$A$ is dense in $\delta M$''.
 In the light of Example \ref{exa_disk} and Corollary \ref{cor_densite_nonnormal}, we thus understand the only problem for $\cc_0^\infty(\omd)$ to be dense in  $\ker \tra_{\pa \omd}$ if $\omd$ is an open subanalytic subset of $\R^n$ (for $p$ large): there might be points of $\delta \omd$ at which this set has dimension less than $(n-1)$.
This  leads us to the following natural notion.

\begin{dfn}\label{dfn_fenced}
 We say that an open subset $\omd$ of $\R^n$   is {\bf fenced} if $\delta \omd=\adh{\pa \omd}$.
\end{dfn}
When this occurs, by Corollary \ref{cor_densite_nonnormal}, we have for all $p$ sufficiently large:
  \begin{equation}\label{eq_fenced}
                                                                                                        W^{1,p}(\omd,\pa \omd)= W_0 ^{1,p}(\omd).
                                                                                                       \end{equation}

\begin{rem}\label{rem_fenced}For an open bounded subset $\omd$ of $\R^n$, the set
\begin{equation}\label{eq_omega_tilde}\omt:=\adh{\Omega}\setminus \adh{\pa \omd}\end{equation} is an open bounded definable subset of $\R^n$ which is always  fenced. In addition,   $\Omega$ is  fenced if and only if   $\omt=\Omega$. In particular, $\tilde{\omt}=\omt$. 
\end{rem}

\end{subsection}
\end{section}


\begin{section}{Some preliminary local estimates}
In this section, we will be working in the setting of section \ref{sect_lcs_mba}, with $\ep$, $r_s$, $M^\eta$, and $N^\eta$, as defined there (see (\ref{eq_metaneta})).
\begin{subsection}{The retraction $R_t^\eta$.} Let us remark that the definition of $r_s(x)$ that we gave in Theorem \ref{thm_local_conic_structure} actually makes sense for each $(s,x)\in [0,\infty)\times \mep$ satisfying $s\le \frac{\ep}{|x|}$. We therefore can define a mapping $R:P \to M$, where $P=\{(t,x)\in \R\times M: 1\le t\le \frac{\ep}{|x|} \}$, by
$$R(t,x):=H(tH^{-1}(x)).$$
We will denote by $R^\eta:[1,\frac{\ep}{\eta}] \times N^\eta \to M$ and $R^\eta_t: N^{\eta} \to N^{t\eta}$ the respective restrictions of $R$ and $R(t,\cdot)$.

We do not use the notation $r(s,x)$ when $s>1$,  since the properties of $R$ will be very different from the properties of $r$ listed in  Theorem \ref{thm_local_conic_structure}. 
The mapping $R$ should rather be regarded as an inverse mapping, as it indeed directly follows from the definitions that the mapping $R_t^\eta:N^\eta\to N^{t\eta}$, induced by $R$,  is the inverse of the mapping $r_{\frac{1}{t}}^{t\eta}: N^{t\eta}\to N^{\eta}$ induced by $r$. 
Although $R$ is not Lipschitz, we have the following estimates of its first derivative.
 \end{subsection}

\begin{subsection}{Some properties of $R_t^\eta$.} There is a constant $C$ such that for all $\eta\in (0,\ep)$, we have for all $(t,x)\in [1,\frac{\ep}{\eta}] \times N^\eta$:

\begin{equation}\label{eq_jad_reta}
 \jac R_t^\eta(x) \ge \frac{t^{m-1}}{C},
\end{equation}

\begin{equation}\label{eq_par}
 \left|\frac{\pa R^\eta}{\pa t}(t,x)\right| \le C\eta.
\end{equation}

\begin{proof}
 Since for all $s\in (0,1)$ and $\eta\in (0,\ep)$, the induced mapping $r_s^\eta: N^\eta\to N^{s\eta}$ is $Cs$-Lipschitz for some constant $C$ independent of $s$ and $\eta$, we have \begin{equation}\label{eq_jac_1}\jac r_s^\eta\le C^{m-1} s^{m-1}. \end{equation} Hence, for $t\in [1,\frac{\ep}{\eta}]$, we can write for $x\in N^\eta$: $$\jac R^\eta_t(x)=\jac (r_{\frac{1}{t}} ^{t\eta})^{-1}(x)=\frac{1}{\jac r_{\frac{1}{t}} ^{t\eta}\big(  (r_{\frac{1}{t}}^{t\eta})^{-1}(x)\big)}\overset{(\ref{eq_jac_1})}{\ge} \frac{t^{m-1}}{C^{m-1}},$$
which yields (\ref{eq_jad_reta}).
To show the second inequality, observe that, by definition of $R$, we have $\frac{\pa R}{\pa t}(t,x)=D_{tH^{-1}(x)} H(H^{-1}(x))$, where $H$ is provided by Theorem \ref{thm_local_conic_structure} (at $x_0=0$). As a  matter of fact, (\ref{eq_par}) immediately comes down from the facts that $H^{-1}$ preserves the distance to the origin and $H$ has bounded derivative.
\end{proof}
\end{subsection}

\begin{subsection}{Local estimates of the $L^p$ norm.}\label{sect_estimates_loc}
 In \cite{trace},  local estimates  were obtained in the case where $p$ is large by integrating along the trajectories of $r_s$ (see (\ref{theta_cont}) below). We are now going to make use of the isotopy $R_t$ that we just introduced in a similar way so as to get estimates that are valid for all $p$. 
\begin{lem}\label{lem_u_neta} For every  $p\in[1,\infty)$ there is a constant $C$  such that for all $u\in \cc^1(M)
$  with $\supp \,u\subset \bou(0,\ep)$ as well as $|\pa u| \in L^p(M)$, and all $\eta\in (0,\ep]$,  we have:\begin{enumerate}[(i)]   
\item  If $p<m$ then
$$ ||u||_{L^p(N^\eta)}\le  C\eta^{\frac{p-1}{p}}\left(\eta^{\frac{m-p}{4p}}||\pa u||_{L^p(M)}+||\pa u||_{L^p(M^{\eta^{1/2}})}\right)  .$$
\item  If $p=m$ then \begin{equation*}\label{eq_neta_pgd}
                                      ||u||_{L^p(N^\eta)}\le C\cdot \eta^{\frac{m-1}{m}} \cdot  \big(\ln \frac{1}{\eta}\big)^{\frac{m-1}{m}} \cdot ||\pa u||_{L^p(M)},
                                    \end{equation*}
                              \item If $p>m$ then \begin{equation*}\label{eq_neta_pgd}
                                      ||u||_{L^p(N^\eta)}\le C\cdot \eta^{\frac{m-1}{p}}\cdot ||\pa u||_{L^p(M)}.
                                    \end{equation*}
                           
                             \end{enumerate}
\end{lem}

\begin{proof}We first focus on $(i)$. By the fundamental theorem of calculus, we have
for all $p\in [1,\infty)$,  $u\in  \cc^1(M)
$  with $\supp \,u\subset \bou(0,\ep)$ and $|\pa u| \in L^p(M)$, as well as $\eta\in (0,\ep]$, and $x\in N^\eta$:
$$u(x)=-\int_1 ^{\ep/\eta} \frac{\pa (u\circ R^\eta)}{\pa t}(t,x)dt,$$
from which it follows that:
\begin{eqnarray}\label{eq_u_eta_prel}
||u||_{L^p(N^\eta)} &\le&\int_1 ^{\ep/\eta} \left(\int_{N^\eta}\left|\frac{\pa (u\circ R^\eta)}{\pa t}\right|^p(t,x)dx\right)^{1/p}dt \qquad\mbox{(by Minkowski's inequality)}\nonumber\\
 &=&\int_1 ^{\ep/\eta} t^{-\frac{l}{p}}\left(\int_{N^\eta}\left|\frac{\pa (u\circ R^\eta)}{\pa t}\right|^p(t,x)t^ldx\right)^{1/p}dt,\quad  \mbox{with $l:=\frac{m+p}{2}-1$},\nonumber\\
 &\le&\left(\int_1 ^{\ep/\eta} t^{-\frac{lp'}{p}}dt\right)^{1/p'}\left(\int_1 ^{\ep/\eta} \int_{N^\eta}\left|\frac{\pa (u\circ R^\eta)}{\pa t}\right|^p(t,x)t^{l}\,dx\,dt\right)^{1/p}, \end{eqnarray}
 by H\"older's inequality (for $p>1$; if $p=1$ then simply notice that $t^{-l/p}\le 1$ since $l\ge 0$). As, for $p\in (1,m)$, we have\begin{equation*}\label{eq_lp_sur_prime}\frac{lp'}{p}=\frac{\frac{m+p}{2}-1}{p-1}>1,\end{equation*}                                                                                                                                                        
  raising the preceding estimate to the power $p$, we get for $p\in [1,m)$:
 \begin{eqnarray}\label{eq_u_neta}
||u||_{L^p(N^\eta)}^p\nonumber
  &\lesssim&\int^{\ep/\eta} _{1}\int_{N^\eta} \left|\frac{\pa (u\circ R^\eta)}{\pa t}(t,x)\right|^p t^ldx\,dt\\\nonumber
   &\le& \int^{\ep/\eta} _{1}\int_{N^\eta}  |\pa u(R^\eta_t(x))|^p \left|\frac{\pa R^\eta}{\pa t} (t,x)\right|^p  t^ldx\,dt\\\nonumber
           &\overset{(\ref{eq_par})}{\lesssim}&\eta^{p}\int^{\ep/\eta} _{1}t^{l}\int_{N^\eta} |\pa u(R^\eta_t(x))|^p dx\,dt\\\nonumber
   &\overset{(\ref{eq_jad_reta})}{\lesssim}&\eta^{p}\int^{\ep/\eta} _{1}t^{l-m+1}\int_{N^\eta} |\pa u(R^\eta_t(x))|^p \, \jac R^\eta_t(x)\,dx\,dt\\
                &=&\eta^{p}\int^{\ep/\eta} _{1}t^{l-m+1}||\pa u||_{L^p(N^{t\eta})}^pdt.
 \end{eqnarray}
                                     Observe now that $l-m+1=\frac{p-m}{2}$, which is negative for $p$ in $[1,m)$, so that:
 \begin{equation}\label{eq_pierwsze}\int^{\eta^{-1/2}} _{1}t^{\frac{p-m}{2}}||\pa u||_{L^p(N^{t\eta})}^pdt \le  \int^{\eta^{-1/2}} _{1}||\pa u||_{L^p(N^{t\eta})}^pdt\overset{(\ref{eq_coarea_sph})}{\lesssim}\frac{1}{\eta} ||\pa u||_{L^p( M^{\eta^{1/2}})}^p\;.\end{equation} 
  Moreover, 
if $t\ge \eta^{-1/2}$ then $t^{\frac{p-m}{2}}\le  \eta^{\frac{m-p}{4}}$ (since $\frac{p-m}{2}<0$), which entails that for $\eta^{-\frac{1}{2}}<\frac{\ep}{\eta}$:  \begin{equation*}\label{eq_drugie}
\int^{\ep/\eta} _{\eta^{-1/2}}t^{\frac{p-m}{2}}||\pa u||_{L^p(N^{t\eta})}^pdt\le \eta^{\frac{m-p}{4}} \int^{\ep/\eta} _{\eta^{-1/2}}||\pa u||_{L^p(N^{t\eta})}^pdt \overset{(\ref{eq_coarea_sph})}{\lesssim} \eta^{\frac{m-p}{4}-1} ||\pa u||_{L^p(M)}^p.\end{equation*}
                                Together with (\ref{eq_u_neta}) and (\ref{eq_pierwsze}),  this yields $(i)$. In order to show now $(iii)$, observe first that for all $p>m$ we have:   
\begin{equation}\label{eq_tl}
\left( \int_1 ^{\ep/\eta} t^{-\frac{(m-1)p'}{p}}dt\right)^{1/p'}\lesssim \eta^{\frac{m}{p}-1},
\end{equation}
so that, by (\ref{eq_u_eta_prel}) for $l=m-1$, we get
\begin{eqnarray*}\label{eq_u_neta_2}
||u||_{L^p(N^\eta)}^p
 \lesssim \eta^{m-p}\int^{\ep/\eta} _{1}\int_{N^\eta} \left|\frac{\pa (u\circ R^\eta)}{\pa t}(t,x)\right|^p t^ldx\,dt.
\end{eqnarray*}
This double integral can be estimated by the computation carried out in (\ref{eq_u_neta}) (still for $l=m-1$).  We thus obtain: 
\begin{eqnarray*}\label{eq_u_neta_2}
 ||u||_{L^p(N^\eta)}^p\overset{(\ref{eq_u_neta})}{\lesssim} \eta^{m-p}\cdot \eta^p \int^{\ep/\eta} _{1}||\pa u||_{L^p(N^{t\eta})}^pdt \overset{(\ref{eq_coarea_sph})}{\lesssim} \eta^{m-1}||\pa u||_{L^p(M)}^p,
\end{eqnarray*}
as required. To show $(ii)$, just replace (\ref{eq_tl}) in the proof of $(iii)$ (assuming now $p=m$) with (for $m>1$):
$$\left( \int_1 ^{\ep/\eta} t^{-\frac{(m-1)p'}{p}}dt\right)^{1/p'}\lesssim  \; \big(\ln \frac{1}{\eta}\big)^{\frac{m-1}{m}}  , $$
and proceed in the same way (if $p=m=1$, simply notice that $t^{-(m-1)}\equiv 1$).
\end{proof}

\begin{rem}\label{rem_neta}\begin{enumerate}[(a)]
\item In $(i)$ and $(iii)$ of the lemma, the constant $C$  can be bounded independently of $p$ if this number stays bounded away from $m$.


\item Compiling together $(i)$, $(ii)$, and $(iii)$ of the above lemma, we can get that for any $p\in [1,\infty)$ we have for  $u\in \cc^1(M)
$  satisfying $\supp \,u\subset \bou(0,\ep)$ and $|\pa u|\in L^p(M)$, as well as $\eta\le \ep$:
\begin{equation}\label{eq_link_forallp}
                                        ||u||_{L^p(N^\eta)}\lesssim  \eta^{\frac{a-1}{p}}\cdot \big(\ln \frac{1}{\eta}\big)^{\frac{m-1}{m}}\cdot ||\pa u||_{L^p(M)},
\end{equation}
with $a:=\min(m,p)$. 
\end{enumerate}

\end{rem}

\medskip

  \begin{exa}
 The estimate provided by $(ii)$ of the latter lemma is sharp in the sense that we cannot lower the exponents. Let $M:=\{(x,y)\in \R^2:y>x>0\}$ and let $u(x,y):=\ln^\alpha \frac{1}{y}$. We have $\frac{\pa u}{\pa y}(x,y)=\frac{-\alpha}{y\ln^{1-\alpha}  \frac{1}{y}}$ so that $||\pa u||_{L^2(N^\eta)}\in L^2([0,\ep])  $ if and only if $\alpha<\frac{1}{2}$, while for $\eta>0$
 $$\eta^{\frac{1}{2}} \ln^{\alpha} \frac{1}{\eta} \lesssim ||u||_{L^2(N^\eta)}  .$$
 \end{exa}

\medskip

\begin{lem}\label{lem_umeta} For every $p\in [1,\infty)$, there is a constant $C$ such that for
    any distribution $u$ on $M$ satisfying $\supp \,u\subset \bou(0,\ep)$ and $|\pa u|\in L^p(M)$, we have for all $\eta\in (0,\ep]$:
    \begin{enumerate}
       \item     If $p\in [1,m)$ then
\begin{equation}\label{eq_m_eta_ppt}
||u||_{L^p(M^\eta)}\le C \, \eta \left(\eta^{\frac{m-p}{4p}}||\pa u||_{L^p(M)}+||\pa u||_{L^p(M^{\eta^{1/2}})}\right)  .
\end{equation}  
\item  If $p=m$ then  
\begin{equation}\label{eq_m_eta_pm}
                                       ||u||_{L^p(M^\eta)}\le C \, \eta \cdot  \big(\ln \frac{1}{\eta}\big)^{\frac{m-1}{m}}\cdot ||\pa u||_{L^p(M)},
\end{equation}
\item If $p>m$ then
\begin{equation}\label{eq_m_eta_pgd}
                                       ||u||_{L^p(M^\eta)}\le C \, \eta^{\frac{m}{p}}\cdot ||\pa u||_{L^p(M)}.
\end{equation}                                    \end{enumerate}
\end{lem}

\begin{proof}
 If $u$ is a $\cc^1$ function, it suffices to integrate with respect to $\eta$ the estimates of the preceding lemma raised to the power $p$, and to apply (\ref{eq_coarea_sph}).

 To prove the result in the case where $u$ may fail to be smooth,  note that such a distribution $u$ is an $L^p_{loc}$ function \cite[Section 1.1.2]{ma}. Using some regularization operators on manifolds \cite{derham}, we can construct a sequence $u_{i}\in \cc^\infty(M)$ such that $u_i$ tends to $u$ in the $L^p_{loc}$ topology and $\pa u_i$ tends to $\pa u$ in the $L^p$ norm (the $L^p$ convergence of the regularizations is proved in \cite{gold}).  The desired estimates then follow from Fatou's Lemma.
\end{proof}

 In the case where the function $u$ tends to zero at the origin and $p$ is large,  inequality (\ref{theta_cont}) below will provide an estimate which is better than (\ref{eq_m_eta_pgd}). The advantage of (\ref{eq_m_eta_pgd}) is on the other hand to be valid for all $p>m$.

\begin{rem}\label{rem_umeta}   As in Remark \ref{rem_neta}, compiling together these estimates,  we  get that for any $p\in [1,\infty)$, we have for distributions  $u$ on $M$
  satisfying $\supp \,u\subset \bou(0,\ep)$ and $|\pa u|\in L^p(M)$, and $\eta\le \ep$:
 \begin{equation}\label{eq_link_forallpm}
                                         ||u||_{L^p(M^\eta)}\lesssim  \eta^{\frac{a}{p}}\cdot \big(\ln \frac{1}{\eta}\big)^{\frac{m-1}{m}}\cdot ||\pa u||_{L^p(M)},
 \end{equation} 
 where, as in the latter remark,  $a:=\min(m,p)$.
\end{rem}

We end this section with an estimate devoted to the large values of $p$ ((\ref{eq_u_circ_r_sm}) below) that we
 shall need to establish the continuity of Morrey's embedding in section \ref{sect_morrey}.

\begin{lem}
There is a constant $C$ such that for all $p\in [1,\infty)$  sufficiently large and all $u\in W^{1,p}(M^\ep)$ satisfying $\tra_{\{0_{\R^n}\}}u=0$ we  have for almost every $\eta<\ep$:
\begin{equation}\label{theta_cont}|| u||_{L^p(N^\eta)}\lesssim \eta^{\frac{p-1}{p}}\,||u||_{W^{1,p}(M^\eta)}.\end{equation}
\end{lem}
\begin{proof}This inequality follows from $(2.8)$ of \cite{trace}, where it is stated with the operator $\Theta^M$ introduced there. If $\tra_{\{0_{\R^n}\}}u=0$ then $u=\Theta^M u$ almost everywhere, by Lemma $2.4$ of the latter article. \end{proof}

Inequalities (\ref{theta_cont}) and  (\ref{eq_coarea_sph}) yield that for all $p$ large enough,   $u\in W^{1,p}(\mep)$ satisfying  $\tra_{\{0_{\R^n}\}}u=0$, and  $\eta\in (0,\ep)$, we have:
\begin{equation}\label{theta_contm}|| u||_{L^p(M^\eta)}\lesssim\eta ||u||_{W^{1,p}(M^\eta)}.\end{equation}
This estimate enables us to establish:

\begin{lem}
 There is a constant $C$ such that for all $p$ sufficiently large and all $u\in W^{1,p} (\mep)$ satisfying $\tra_{\{0_{\R^n}\}} u=0$, we have for all $\eta\le \ep$ and all $s\in (0,1)$:
  \begin{equation}\label{eq_u_circ_r_sm}|| u\circ r_s||_{W^{1,p} (M^\eta)}\le C s^{1-\frac{\nu}{p}}||u||_{W^{1,p}(M^{s\eta})} ,\end{equation}  where $\nu$ is given by (\ref{eq_jacr_s}).
\end{lem}
\begin{proof}
 For  $\eta\in [ 0,\ep]$ and almost every $s\in (0,1)$, we have for all  $u\in W^{1,p} (\mep)$:
  \begin{eqnarray}\label{eq_pa_u_rond_r}
\nonumber||\pa  (u\circ r_s)||_{L^p(M^\eta)}^p&=&\int_{M^\eta} \left|^{\bf t}D_x r_{s}\left(\pa  u(r_s(x))\right)\right|^p dx\lesssim\int_{M^\eta} \left|\pa  u(r_s(x))\right|^p s^p dx
\\
&\overset{(\ref{eq_jacr_s})}{\lesssim}&\int_{M^\eta} \left|\pa  u(r_s(x))\right|^p\cdot \jac r_{s}^\eta(x)\cdot s^{p-\nu} dx\lesssim s^{p-\nu} ||\pa  u||_{L^p(M^{s\eta})} ^p.
\end{eqnarray}
Moreover, whenever $u$ satisfies $\tra_{\{0_{\R^n}\}} u=0$, we have for $p$ large:
 \begin{equation*}|| u\circ r_s||_{L^p (M^\eta)}\overset{(\ref{lnm})}{\lesssim}
  s^{-\frac{\nu}{p}}||u||_{L^p(M^{s\eta}) }  \overset{(\ref{theta_contm})}{\lesssim}s^{1-\frac{\nu}{p}}\cdot \eta\cdot ||u||_{W^{1,p}(M^{s\eta})} .\end{equation*}
 Together with (\ref{eq_pa_u_rond_r}), this yields  (\ref{eq_u_circ_r_sm}). 
\end{proof}

 \end{subsection}
 \end{section}

\begin{section}{Density of smooth functions and duality between Sobolev spaces}\label{sect_density_and_duality}
We focus in this section on the ``small'' values of $p$ and partially generalize the results of \cite{trace}. We then derive duality results between $W^{1,p}(\omd, \pa \omd)$ and $W^{-1,p'}(\omd)$.

\begin{subsection}{Density of smooth functions.}\label{sect_density} We adopt the convention $\dim \emptyset=-\infty$.

\begin{thm}\label{thm_dense_lprime}Let $Z$ be a definable open subset of $\pau M$ and let $A$ be a definable subset of $ \delta M$ containing $\delta M\setminus \pau M$ and $\delta Z$, with   $\dim A\le m-2$. If $M$ is connected along $\pau M\setminus (Z \cup \adh A )$  then for all $p\in [1,m- \dim A]$ not infinite,    $\cc^\infty_{\mba\setminus (Z\cup \adh{A})}(\mba)$ is dense in $W^{1,p}(M,Z)$.
\end{thm}

\begin{proof} If  $A$ is empty then so is $\delta M\setminus \pau M$ and $Z$ is closed (and open in $\delta M$ by assumption), which means that $\mba$ is a Lipschitz manifold with boundary $\pa M$ at each point of this set,  and a finite union of manifolds with boundary at each point of $Z$, in which cases it is well-known that we can find such approximations (as we can use a partition of unity, it is enough to prove that we can approximate the germ of a function $u\in W^{1,p}(M,Z)$ at a point of $\delta M$). 

We thus assume $A$ to be nonempty and fix  a locally bi-Lipschitz trivial stratification $\Sigma$ of $\adh{M}$ compatible with $\delta M, Z,\adh{A}$, and $ \pau M$. 
 Given $k\in \{0,\dots,c\}$, where $c:=\dim A$, let $A_k$ denote the union of all the strata $S\subset \adh{A}$ of $\Sigma$ satisfying $\dim S\le k$, and set $A_{-1}:=\emptyset$.
  For $u\in W^{1,p}( M)$, let us    define \begin{equation}\label{eq_kappa_u}\kappa_u: =\max\left\{k\in\{0,\dots, c+1\}\,:A_{k-1}\cap \supp_\mba u=\emptyset\right\}.\end{equation}
 We will prove that we can approximate such a function $u$ by decreasing induction on $\kappa_u\in \{0,\dots,c+1 \}$. 
 
  Assume first $\kappa_u=c+1$, take $u\in W^{1,p}(M,Z)$,  and notice that in this case,    $\supp_\mba u$ is a compact subset of  $\mba\setminus A_{c}  $. It suffices to show that every point $\xo$ of $\adh{\supp u}\cap \delta M\subset \mba\setminus \adh A$ has a neighborhood on which $u$ admits approximations. If $\xo\in Z$, then $\mba$ is, in the vicinity of $\xo$, a finite union of Lipschitz manifolds with common boundary $Z$, in which case, since $u$ has zero trace on $Z$, the result is well-known \cite[end of Theorem III.2.46]{boyer}. If $\xo \in \delta M\setminus Z$, then as $\xo\notin \adh A$, which contains $\delta M\setminus \pau M$ and $\delta Z$, we must have $\xo \in \pau M\setminus (\adh Z\cup \adh A)$, and our assumptions entail that $M$ is connected at $\xo$, which means that $\mba$ is a Lipschitz manifold with boundary at $\xo$, and that the result is again well-known.

 To perform the induction step, we fix an integer $k\le c$, assume that the claimed statements hold for all $v\in  W^{1,p}( M)$ satisfying $\kappa_v>k$, and take a function $u\in  W^{1,p}( M)$ with $\kappa_u\ge k$. 
 
 Every point $x_0\in \overline{M}$ admits a neighborhood $U_{x_0}$ in $\overline{M}$ on which we have a bi-Lipschitz trivialization of $\Sigma$.  As we can use a partition of unity, we may assume that $\supp_\mba u$ fits in $U_{x_0}$, for some $x_0\in \adh{M}$.

Let $S$ denote the stratum of $\Sigma$ that contains the point $x_0$.   Observe that if $\dim S< \kappa_u$ then the function $u$ is zero near $S$, so that in such a situation, choosing $U_{x_0}$ smaller if necessary, we can require the function $u$ to be zero on $U_{x_0}$. We thus will suppose $\kappa_u\le \dim S$.

Moreover, as $\supp_\mba u \subset \uxo$ which does not meet any stratum of dimension less than $\dim S$, we see that if $\dim S$ is bigger than $k$ then $\kappa_u>k$, and the result follows from the induction hypothesis.  We thus can assume $\dim S\le k$.

We conclude from the reductions of the two above paragraphs that we can assume $\dim S \le k\le \kappa_u \le \dim S$, which means that it is no loss of generality to only deal with the case $\dim S=\kappa_u =k$. We will do the proof for $x_0=0_{\R^n}\in \delta M$. 
 
   
  By definition of bi-Lipschitz triviality, there is a bi-Lipschitz homeomorphism $\Lambda: U_0 \to (\pi^{-1}(0_{\R^n})\cap \adh{M})\times W_0$ (here $U_0=U_{0_{\R^n}}$), where $\pi:U_0 \to S$ is a $\cc^\infty$ definable (see Remark \ref{rem_pi_S} (\ref{item_retractions})) retraction and $W_0$ is a neighborhood of the origin in $S$ (here $k$ may be $0$, in which case $\Lambda$ is the identity).
  
     Let \begin{equation}\label{eq_FV}F^\eta:=\pi^{-1}(0_{\R^n}) \cap M^\eta \et V^\eta:=F^\eta\times W_0 .\end{equation}
 We may assume the set $F^\ep$, $\ep>0$ small, to be a $\cc^\infty$ manifold of dimension $(m-k)\ge (m-c)$ (see Remark \ref{rem_pi_S} (\ref{item_whitney})). We postpone the case $p=m-c$, assuming for the moment $p<m-c$.
 
 Thanks to (\ref{eq_m_eta_ppt}) for  the manifold $F^\ep$, we thus see that for $v\in W^{1,p}(F^\ep)$ with $\supp\, v\subset \bou(0,\ep)$, $\eta\le \ep$, 
 we have (as $p\in [1,m-c)\subset [1,m-k)$):
  $$||v||_{L^p(F^\eta)} \lesssim \eta\left(\eta^{\frac{m-k-p}{4p}}||v||_{W^{1,p}(F^\ep)}+||v||_{W^{1,p}(F^{\eta^{1/2}})}\right). $$
 For simplicity, we will identify $U_{0}\cap M$ with its image under $\Lambda$ (which is a bi-Lipschitz homeomorphism) and work as if $u$ were a function on this set.
 Given $y\in W_0$ and $x\in F^\ep$, set $v^y(x):= u(x,y)$. Applying the just above inequality to $v^y$, raising it to the power $p$,  and integrating with respect to $y\in W_0$, we easily derive that  for $\eta\le \ep$ and $p\in [1,m-c)$ (assuming $u\in W^{1,p}(V^\ep)$ and $\adh{\supp\, u}\subset \bou(0,\ep)\times W_0$):
    \begin{equation}\label{eq_u_pr_cutoff}||u||_{L^p(V^\eta)} \lesssim \eta\left(\eta^{\frac{m-k-p}{4p}}||u||_{W^{1,p}(V^\ep)}+||u||_{W^{1,p}(V^{\eta^{1/2}})}\right). \end{equation}
       Let now $\psi:\R \to [0,1]$  be a $\cc^\infty$ function  
such that $\psi\equiv 0$ on $(-\infty,\frac{1}{2})$ and $\psi\equiv 1$ near $[1,+\infty)$, and set for $\eta$ positive and $(x,y)\in V^\ep$, \begin{equation}\label{eq_psieta}\psi_\eta(x,y):=\psi\left(\frac{|x|^2}{\eta^2}\right).\end{equation}

We claim that $v_\eta:=\psi_\eta u$ tends to $u$ in $W^{1,p}(M)$ as $\eta\to 0$. Clearly, $||v_\eta- u||_{L^p(M)}$ and $||(\psi_\eta-1) \pa u||_{L^p(M)}$ both tend to zero as $\eta$ tends to zero.  Consequently, if we write:
\begin{equation}\label{eq_udpsi_minus} ||\pa (\psi_\eta u)-\pa u||_{L^p(M)} \le   ||u\, \pa \psi_\eta ||_{L^p(M)}+ ||(\psi_\eta- 1)\,\pa u||_{L^p(M)},\end{equation}                                                                                                                                             
we see that it suffices to check that $||u\, \pa \psi_\eta ||_{L^p(M)}$ tends to zero. But since $\sup |\pa \psi_\eta|\lesssim \frac{1}{\eta}$ and $\supp \, |\pa \psi_\eta| \subset V^{\eta}$, we have:
\begin{equation}\label{eq_udpsi} ||u\, \pa \psi_\eta ||_{L^p(M)}\lesssim \; \frac{||u ||_{L^p(V^\eta)}}{\eta},\end{equation}
which, due to (\ref{eq_u_pr_cutoff}), must tend to zero as $\eta\to 0$, which yields our claim.
As $\kappa_{v_\eta}>k$, we now see that the result (in the case $p<m-c$) follows from our induction hypothesis (since $\tra_{\pau M} v_\eta=\psi_\eta\cdot \tra_{\pau M} u$).

It therefore only remains  to carry out the induction step in the case $p=m-c$ (assuming $A\ne \emptyset$, since otherwise $p=+\infty$ which is excluded in the statement of the theorem). We will rely on the same argument, just replacing the  above  function $\psi_\eta$ with an adequate function $\psi'_\eta$, and the proof that we are going to perform actually works for each $p\in (1,m-c]$.

 We thus first construct, for every $\eta>0 $ small, a function $\psi'_\eta  \in W^{1,\infty}(\R^n)$ such that 
\begin{enumerate}\item $\psi_\eta'\equiv 0$ near the origin  and   $\psi_\eta'\equiv 1$ on $\R^n \setminus \bou(0,\eta)$. \item For almost all $\zeta\in (0,\eta]$  \begin{equation}\label{eq_der_cutoff}
||\pa \psi'_\eta||_{L^\infty(\sph (0,\zeta))}\le \frac{1}{\zeta\cdot \ln \frac{1}{\zeta}\cdot \ln\left(\ln \frac{1}{\zeta}\right) }         .                                                                                                                                                                  \end{equation}\end{enumerate}      
 
 Let for this purpose, given $\zeta \in (0,\frac{1}{e})$, $$g(\zeta):=\ln\left(\ln\left(\ln \frac{1}{\zeta}\right)\right).$$ As $g$ is a smooth function that tends to infinity at $0$, for every $\eta\in(0,\frac{1}{e})$ there is $a_\eta<\eta$ such that $g(a_\eta)=1+g(\eta)$. Let then
$$\phi_\eta(\zeta)=
\begin{cases}1\quad &\mbox{ if }\; \zeta\ge \eta\\
1+g(\eta)-g(\zeta) \quad &\mbox{ if } \; a_\eta \le  \zeta\le  \eta\\
0 \quad&\mbox{ if }\; \zeta\le a_\eta \;,
\end{cases}$$
and finally set $\psi'_\eta(x):=\phi_\eta(|x|)$. 

The function $\psi'_\eta$ is continuous, piecewise smooth, and constant on every $\sph(0,\zeta)$, $\zeta>0$. It is not difficult to check that $\psi_\eta'\in W^{1,\infty} (\R^n)$ and that  for $a_\eta<|x|<\eta$ we have
\begin{equation}\label{eq_papsi_cone}\pa \psi'_\eta(x)=\phi'_\eta(|x|) \cdot \frac{x}{|x|}=-g'(|x|)\cdot\frac{x}{|x|}=\frac{1}{|x| \cdot \ln \frac{1}{|x|} \cdot \ln \left(\ln \frac{1}{|x|}\right) }\cdot  \frac{x}{|x|},\end{equation}
which clearly yields (\ref{eq_der_cutoff}).

We are now ready to perform the proof in the case $p=m-c$. Let $$G_y^\eta:=\left(\pi^{-1}(0)\cap \sph(0,\eta)\right) \times\{y\} \et F_y^\eta:=F^\eta \times\{y\}.$$
Thanks to $(i)$ and $(ii)$ of Lemma \ref{lem_u_neta} for  the manifold $F^\ep$, we see that for almost every $y\in W_0$ (as $\supp\, u\subset V^\ep$), 
 we have (as $p=m-c\le m-k$) for almost every $\zeta\le \ep$:
  \begin{equation}\label{eq_gy}||u||_{L^p(G_y^\zeta)}^p \lesssim \zeta^{p-1}\cdot \ln ^{p-1} \frac{1}{\zeta} \cdot  ||u||_{W^{1,p}(F_y^\ep)}^p. \end{equation}
 We claim that
the function $v_\eta':=\psi'_\eta u$ tends to $u$ in $W^{1,p}(F^\ep \times W_0)$ as $\eta\to 0$, where we regard the above function $\psi'_\eta: \R^n \to \R$ as a function on $F^\ep\times W_0$, constant with respect to the second variable. As in the case $p<m-c$, it suffices to check that $||u\, \pa \psi_\eta' ||_{L^p(M)}$ tends to zero.  
For almost every $y\in W_0$, by (\ref{eq_der_cutoff}) and (\ref{eq_gy}), we have  for $\eta>0$ small
 \begin{equation}\label{eq_u_pa_psi_eta_N}||u\pa \psi'_\eta||_{L^p(G_y^\zeta)}^p \lesssim \frac{\zeta^{p-1} \cdot  \ln^{p-1} \frac{1}{\zeta}}{\zeta^p  \cdot \ln^p \frac{1}{\zeta} \cdot \ln^p \left( \ln \frac{1}{\zeta}\right) } \cdot  ||u||_{W^{1,p}(F_y^\ep)}^p= \frac{1}{\zeta  \cdot \ln \frac{1}{\zeta} \cdot \ln^p \left( \ln \frac{1}{\zeta}\right) } \cdot  ||u||_{W^{1,p}(F_y^\ep)}^p\end{equation}
 so that, since $\supp\, |\pa \psi'_\eta|  \subset F^\eta\times W_0$, we then can write
 \begin{eqnarray*}||u\pa \psi'_\eta||_{L^p(F_y^\ep)}^p &\overset{(\ref{eq_coarea_sph})} \lesssim & \int_0 ^\eta||u\pa \psi'_\eta||_{L^p(G_y^\zeta)}^pd\zeta
 \\ & \overset{(\ref{eq_u_pa_psi_eta_N})}\lesssim& \int_0 ^\eta \frac{  d\zeta}{\zeta \cdot \ln \frac{1}{\zeta}\cdot \ln^p \left( \ln \frac{1}{\zeta}\right)  } \cdot ||u||_{W^{1,p}(F_y^\ep)}^p
 \\ &\lesssim& \ln^{1-p}\left(\ln \frac{1}{\eta}\right)\cdot  ||u||_{W^{1,p}(F_y^\ep)}^p \; .\end{eqnarray*}
 As the constant of this estimate is independent of $y$, we get after integrating with respect to $y\in W_0$ for $\eta>0$ small
  $$||u\pa \psi'_\eta||_{L^p(M)}^p \lesssim \ln^{1-p}\left(\ln \frac{1}{\eta}\right)\cdot  ||u||_{W^{1,p}(F^\ep\times W_0)}^p,$$
 which tends to zero as $\eta$ tends to zero (as $p>1$),  yielding our claim.

Consequently, as $\kappa_{v'_\eta}>k$, we can see that the result follows from our induction hypothesis (since $\tra_{\pau M} v'_\eta=\psi'_\eta\cdot \tra_{\pau M} u$).
\end{proof}

Let now $$\pms:=m-\dim\, (\delta M\setminus  \pau M) \overset{(\ref{eq_pa1})}{\ge}2.$$
with, as above, the convention $\dim \emptyset=-\infty$.

 Observe now that, thanks to (\ref{eq_pa1}), Theorem \ref{thm_dense_lprime} gives in the case $A=\delta M\setminus  \pau M$ and $Z=\pau M$ or $\emptyset$:
\begin{cor}\label{cor_densite_A_vide}   For all $p\in [1,\pms]$  (not infinite), the following assertions are true:               
\begin{enumerate}[(i)]                                                                                                                                                                  
\item\label{item_avide_ker}  $\cc_0^\infty(M)$ is dense in $W^{1,p}(M,\pa M) $, which amounts to  $$W^{1,p}(M,\pa M)=W^{1,p}_0(M).$$ 
 \item \label{item_avide_normal}   If $M$ is connected along $\pa M$ then $\cc^\infty_{M\cup \pau M}(\mba)$ is dense in $W^{1,p}(M)$.
 \end{enumerate}   
\end{cor}
\begin{rem}
 It is easy to see that the conclusion of (\ref{item_avide_normal}) fails as soon as $M$ is not weakly normal (see the proof of Corollary $3.10$ of \cite{trace}). 
\end{rem}
We now are going to give consequences of these results in the case where the underlying manifold is an open set $\Omega$.
 \begin{rem}\label{rem_omh} We recall that $\omt=\adh\omd\setminus \adh{\pa \omd}$ (see Remark \ref{rem_fenced}). Since $\dim \omt\setminus \omd<n-1$,  this set is ``removable'', i.e.  $ W^{1,p}(\omd)=W^{1,p}(\omt)$ for all $p\in [1,\infty)$ (see for instance \cite[section 1.1.18]{ma}). As the inclusion map $ W^{1,p}(\omd)\hookrightarrow W^{1,p}(\omt)$  identifies the respective kernels of the trace operators, we deduce that:
 $$ W^{1,p}(\omd,\pau \omd)= W^{1,p}(\omt,\pau \omt).$$
 A proof of these facts (for all $p\in [1,\infty)$)  can also be carried out by applying  Theorem \ref{thm_dense_lprime} (with $p=1$) to a $\cc^\infty$ normalization of $\omd$.
 Note also that Corollary \ref{cor_densite_A_vide} and (\ref{eq_fenced}) then entail that we have for all $p$ in $[1,\pom]$ or sufficiently large  (not infinite):
$$ W^{1,p}(\omd,\pau \omd)= W_0^{1,p}(\omt).$$
\end{rem}
This remark, and Corollary \ref{cor_densite_A_vide} $(\ref{item_avide_ker})$, yield:
\begin{pro}\label{pro_extension}
\begin{enumerate}
 \item\label{item_ext_by_0} For each $p\in [1,\pom]$ or sufficiently large (not infinite), if $u\in W^{1,p}(\omd,\pa \omd)$ then $\adh{u}$, defined as $\adh{u}(x)=u(x)$ if $x\in \omd$ and $\adh{u}(x)=0$ otherwise, belongs to $W^{1,p}(\R^n)$.  
 \item\label{item_ext_dist} For any $p\in [1,\pom]$ (not infinite), every distribution $T:\cc_0^\infty(\omd)\to \R$ which is continuous in the norm 
  $||\cdot||_{W^{1,p}(\omd)}$ uniquely extends continuously (in this norm) to $\cc_0^\infty(\omt)$.
\end{enumerate}
\end{pro}

 \begin{exa}\label{exa_disk_tilde}
  Proposition \ref{pro_extension} $(\ref{item_ext_dist})$ is not true for $p$ large. For instance, if $\Omega$ coincides with the manifold $M$ introduced in Example \ref{exa_disk} then the Dirac measure concentrated at the origin is bounded for $||\cdot||_{W^{1,p}(\omd)}$ for $p$ large,  is zero on $\cc_0^\infty(\omd)$, but  nonzero on $\cc_0^\infty(\omt)$.
 \end{exa}

\begin{pro}\label{pro_carac_wpa}
For all $p$ in $[1,\pom]$ or sufficiently large  (not infinite), we have  for each function $u\in W^{1,p}(\omd,\pau \omd)$ and   $v\in W^{1,p'}(\Omega)$, and all $i\in \{1,\dots,n\}$:
\begin{equation}\label{eq_lsp}  <u,\pa_i v>=-<\pa_i u,v> .\end{equation}
\end{pro}
\begin{proof}
In the case where $p\in [1,\pom]$, by Corollary \ref{cor_densite_A_vide}  $(\ref{item_avide_ker})$, we see that it indeed suffices to prove this identity for $u\in \cc^\infty_0(\Omega)$  and $v\in \cc^\infty(\Omega)$, for which the result follows after integration by parts. 

In the case where $p$ is large, by Remark \ref{rem_omh}, we can assume that $\Omega=\omt$, or equivalently (see Remark \ref{rem_fenced}), that $\Omega$ is  fenced, so that thanks to (\ref{eq_fenced}), we can also suppose that $u\in \cc^\infty_0(\Omega)$  and $v\in \cc^\infty(\Omega)$, and the result can again be established by integrating by parts. 
 \end{proof}
\begin{rem}
 It is possible to prove that this proposition actually also holds in the case $p=\infty$. 
\end{rem}

\end{subsection}
 
 \begin{subsection}{Duality in Sobolev spaces.}\label{sect_dual} 
As usual, we let for $p\in [1,\infty]$:\begin{equation*}
                      W^{-1,p}(\Omega):=\{T\in \cc_0^\infty(\omd)':\exists\, f_0,\dots, f_n \in L^p(\Omega) , \;\; T=\sum_{i=0}^n\pa_i f_i  \},
                     \end{equation*}
 that we endow with the norms:
$$||T||_{W^{-1,p}(\Omega)}:=\inf  \left(\sum_{i=0}^n ||f_i||_{L^p(\Omega)}^p\right)^{1/p},   \quad \mbox{for each } p\in [1,\infty) ,$$
and  $$||T||_{W^{-1,\infty}(\Omega)}:=\inf  \left(\max_{i\in \{0,\dots,n\}} ||f_i||_{L^\infty(\Omega)}\right),  \quad \mbox{for } p=\infty ,  $$
where the infimum is taken  among all the $f_0,\dots,f_n$ in $L^p(\Omega) $ satisfying $T= \sum_{i=0}^n \pa_i f_i$. 

 Observe that, due to  Proposition \ref{pro_extension} (\ref{item_ext_dist}), 
 we have  for all $p\in [1,\pom]$ (not infinite):
\begin{equation}\label{eq_wm1}
 W^{-1,p'}(\omt)= W^{-1,p'}(\omd).
\end{equation}
These two spaces can however be different for $p$ large (see Example \ref{exa_disk_tilde}). 
\begin{rem}\label{rem_repres_functionals}
  For every $p \in (1,\infty)$, the linear mapping  $\mathbf{A}_p: W^{1,p}(\omd)\to L^{p}(\omd)^{n+1}$ defined as $\mathbf A _p (u):= (u, \pa u)$,  is continuous, injective, and has closed image. Its dual mapping  $\mathbf{A}_p':  L^{p'}(\omd)^{n+1}\to W^{1,p}(\omd)'$ is therefore onto. As a matter of fact, every  continuous linear functional $T:W^{1,p}(\omd)\to \R$ can be written $\alpha+\pa \beta$ for some  $\alpha\in L^{p'}(\omd)$ and $\beta\in  L^{p'}(\omd)^n$.
\end{rem}

\medskip

\subsection*{\bf The morphisms $\Lambda_p$ and $\Lambda'_p$.} As well-known, since the elements of  $W^{-1,p'}(\omd)$ are  by definition  continuous functionals on the space $(\cc_0^\infty(\omd),||\cdot||_{  W^{1,p}(\omd)})$, which is dense in $W_0 ^{1,p}(\Omega)$, we  can define two mappings  $$\Lambda_{p}:W^{1,p}_0 (\Omega)\to W^{-1,p'}(\omd)' \et \Lambda_{p}': W^{-1,p'}(\omd)\to W^{1,p}_0(\Omega)'$$  by setting for $T\in W^{-1,p'}(\Omega)$ and  $u\in W^{1,p}_0(\Omega)$:
 $$\Lambda_{p}u(T)= \Lambda'_p T(u)=T(u). $$
 
The following lemma provides some well-known facts that are not specific to subanalytic domains, and that will be of service to derive our duality result.

\begin{lem}\label{lem_dual}
 For all $p \in (1,\infty)$, 
  $\Lambda_p$ and $\Lambda'_p$ are isomorphisms,  and we have
  \begin{equation}\label{eq_w0}
   W^{1,p}_0 (\omd)=\{ u\in W^{1,p} (\omd) : \forall \, \beta \in W^{1,p'}_\nabla (\omd), \;\; <u,\nabla \beta>= -<\pa u, \beta>\},
  \end{equation}
  where $ W^{1,p'}_\nabla (\omd):=\{ \beta\in L^{p'} (\omd)^n : \nabla \beta \in L^{p'}(\omd)   \}$, $\nabla$ denoting the divergence operator.
\end{lem}
\begin{proof}As  $\Lambda_p$ and $\Lambda'_p$ are  clearly injective, it suffices to show that they are onto. The ontoness of  $\Lambda'_p$ is a direct consequence of  Remark \ref{rem_repres_functionals}.

 Let us show (\ref{eq_w0}). Let $Z_p$ denote the set which stands on the right-hand-side of this equality. As we clearly have $W^{1,p}_0 (\omd)\subset Z_p$, let us focus on the proof of the reversed inclusion, for which we will show that every continuous linear functional $T:W^{1,p}(\omd)\to \R $ that vanishes on $\cc_0^\infty (\omd)$ is identically zero on $Z_p$. By Remark \ref{rem_repres_functionals}, such a functional $T$ can be written $\alpha+\pa \beta$ for  some  $\alpha\in L^{p'}(\omd)$ and $\beta\in  L^{p'}(\omd)^n$. As for every $\varphi \in \cc_0^\infty(\omd)$, we have $<\alpha,\varphi>+<\beta,\pa \varphi>=0$, we see that $\alpha=\nabla \beta$ in the sense of distribution, which means that $\beta \in W^{1,p'}_\nabla (\omd)$, so that
 $$T(u)=<\nabla \beta,u>+<\beta,\pa u>=0, $$
 for every $u\in Z_p$, as required.
 
It remains to check that $\Lambda_p$ is onto. Let $\mathbf{B}_p:L^{p'}(\omd)^{n+1}\to W^{-1,p'}(\omd)$ be the continuous linear mapping defined by $\mathbf B _p (\alpha,\beta):= \alpha +\pa\beta$ for $\alpha\in L^{p'}(\omd)$ and $\beta\in L^{p'}(\omd)^n$. Take $T\in W^{-1,p'}(\omd)'$ and observe that, as the pull-back $\mathbf{B} _p^* T$ is a continuous functional on $L^{p'}(\omd)^{n+1}$, there are $u\in L^p(\omd)$ and $v\in  L^{p}(\omd)^n$ such that
\begin{equation}\label{eq_bstarp}\mathbf{B} _p^*T(\alpha,\beta)=<u,\alpha> + <v,\beta>,\end{equation} for all $\alpha\in L^{p'}(\omd)$ and $\beta\in  L^{p'}(\omd)^n$. Since $\mathbf{B} _p^*T$ vanishes on $\ker \mathbf B _p$, we deduce that $v=\pa u$ in the distribution sense, which means that $u\in W^{1,p}(\omd)$. Moreover, we also see that for all $\beta \in W^{1,p'}_\nabla (\omd)$ we have (as $(\nabla\beta, \beta)\in \ker \mathbf B _p$)
$$<u,\nabla \beta>+<\pa u, \beta>=\mathbf{B} _p^*T(\nabla \beta,\beta)=0,$$
which yields that $u\in Z_p\overset{(\ref{eq_w0})}= W^{1,p}_0 (\omd)$. Note that, since $v=\pa u$, (\ref{eq_bstarp}) establishes that $\mathbf B_p^* T(\alpha,\beta)=\mathbf B_p^*\Lambda_pu(\alpha,\beta)$ for all $\alpha\in L^{p'}(\omd)$ and $\beta\in  L^{p'}(\omd)^n$. Since $\mathbf B_p$ is surjective, this shows that  $T=\Lambda_pu$, yielding that $\Lambda_p$ is onto.
\end{proof}

This leads us to the following result which naturally extends the duality that holds on Lipschitz domains.

\begin{pro}\label{pro_dual}
       For every $p\in (1,\pom]$ (not infinite),   $\Lambda_p$ and $\Lambda'_p$ induce isomorphisms $$W^{1,p}(\omd,\pau \omd)\simeq W^{-1,p'}(\omd)'\et W^{-1,p'}(\omd)\simeq W^{1,p}(\omd,\pau \omd)'.$$
  If $\omd$ is  fenced, this remains true for each $p\in [1,\infty)$ sufficiently large. 
\end{pro}
 \begin{proof}
  In virtue of Lemma \ref{lem_dual}, $\Lambda_p$ and $\Lambda_p'$ are isomorphisms.  The first sentence thus follows from $ (\ref{item_avide_ker}) $ of Corollary \ref{cor_densite_A_vide}, and the second one is a consequence of (\ref{eq_fenced}).
 \end{proof}

 \begin{rem}\label{rem_dual}
 The assumption ``fenced'' in the last sentence of this proposition can be avoided by applying this proposition to $\omt$.
 Indeed, it is worthy of notice that thanks to Remark \ref{rem_omh},  we then get for all $p\in (1,\pom]$ or sufficiently large (not infinite): $$W^{1,p}(\omd,\pau \omd)\simeq W^{-1,p'}(\omt)'\et W^{-1,p'}(\omt)\simeq W^{1,p}(\omd,\pau \omd)'.$$
   \end{rem}

\end{subsection}
\begin{subsection}{Applications the Dirichlet problem.}
As promised, we give an application of the material that we have developed to the study of the Dirichlet problem associated with the Laplace equation.

We denote by $\Delta:W^{1,2}(\omd) \to W^{-1,2}(\omd)$  the linear mapping induced by the Laplace operator $\sum_{i=1}^n \pa_i^2$.
      
      \begin{thm}\label{thm_dir} 
       For every $f\in W^{-1,2}(\omd)$, the equation
       $$\begin{cases}
 \Delta u=f \qquad \mbox{on } \omd,\\
u=0 \qquad \mbox{on } \pa \omd
\end{cases}$$
      has a unique (weak) solution in $W^{1,2}(\omd)$.
      \end{thm}
\begin{proof}
  Since $\pom \ge 2$,   Corollary \ref{cor_densite_A_vide} $(i)$ yields that $f$ induces a functional on $W^{1,2}(\omd, \pa \omd)$.      Observe in addition that the density of $\cc_0^\infty(\omd)$ in $W^{1,2}(\omd, \pa \omd)$ also entails that $||u||_{L^2(\omd)} \lesssim ||\pa u||_{L^2(\omd)}$ for  $u\in W^{1,2}(\omd, \pa \omd)$ (one could also make use here of Proposition \ref{pro_gen_poinc1} below or of the Poincar\'e inequality proved in \cite{poincfried}), which, by Lax-Milgram's Theorem, establishes existence and uniqueness of the solution.
\end{proof} 

\begin{rem}
The problem of the continuity of the solution at boundary points is investigated  in \cite{k1}. Theorem \ref{thm_holder} below is likely to bring new information on this issue.
\end{rem}
\end{subsection}
\end{section}

\begin{section}{Sobolev embeddings and Poincar\'e inequality} In this section, we derive some inequalities that come down from the local estimates of section \ref{sect_estimates_loc}. We start with a generalization of Sobolev's Embedding Theorem:
\begin{thm}\label{thm_embedding}
 There is a positive real number $\mu$ such that for all $p$ and $q$ in $[1,\infty)$ with $p\le q \le p+\mu$, there are $\alpha_{p,q}\in[0,1)$ and  $C>0$ such that for all $u\in W^{1,p}(M)$ we have:
 \begin{equation}\label{ineq}
  ||u||_{L^q(M)} \le C ||u||_{L^p(M)} ^{1-\alpha_{p,q}} \cdot ||u||_{W^{1,p}(M)} ^{\alpha_{p,q}}.
 \end{equation}
Moreover,  the embedding $W^{1,p}(M)\hookrightarrow L^q(M)$ is compact for all such $p$ and $q$, and we have: \begin{equation}\label{eq_alpha_lip}\alpha_{p,q}\le c|p-q|.\end{equation}
              for some constant $c$ independent of $p$ and $q$.                      
\end{thm}

\begin{proof} As $\cc^\infty(M)$ is dense in $W^{1,p}(M)$, it suffices show (\ref{ineq}) for a $\cc^\infty$ function $u$.  The compactness of this embedding will be established afterwards. We argue by induction on $m$, the case $m=0$ being vacuous. 

Since we can use a partition of unity,  it is enough to prove that every point $x_0\in \delta M$ has a neighborhood $\uxo$ such that $(\ref{ineq})$ holds for every function $u\in W^{1,p}(M)$ satisfying $\supp\, u \subset \uxo$, and we can assume $x_0=0$. 

    Let $r_s$, $\ep$, $M^\eta$, and $N^\eta$ (for $\eta\le \ep$), be as in section \ref{sect_lcs_mba}, and apply the induction hypothesis to $N^\ep$. This provides a constant $\tilde{\mu}$ such that for every $p$ and $q$ satisfying $p\le q\le p+\tilde{\mu}$, there is $\tilde{\alpha}_{p,q}\in [0,1)$  such that for all $v\in W^{1,p}(\nep)$:
  \begin{equation}\label{ineq_hr}
  ||v||_{L^q(\nep)} \lesssim ||v||_{L^p(\nep)} ^{1-\tilde{\alpha}_{p,q}} \cdot ||v||_{W^{1,p}(\nep)} ^{\tilde{\alpha}_{p,q}}.
 \end{equation}
 
 Fix $p\in [1,\infty)$ and set $\tilde{p}:=p+\tilde{\mu}$, as well as $\beta:=\alphat_{p,\tilde{p}}$.  Observe now that, as $r$ is Lipschitz, we have for  $u\in \cc^\infty(M)$ and  $s\in (0,1)$:
 \begin{eqnarray*}
  ||u||_{L^{\tilde{p}}(N^{s\ep})}\lesssim  ||u\circ r_s||_{L^{\tilde{p}}(N^{\ep})} \overset{(\ref{ineq_hr})}{\lesssim}  ||u\circ r_s||_{L^p(N^{\ep})}^{1-\beta}  ||u\circ r_s||_{W^{1,p}(N^{\ep})}^{\beta},
 \end{eqnarray*}
which by (\ref{ln}) (applied to both $u\circ r_s$ and $|\pa u|\circ r_s$, as $|\pa (u\circ r_s)|\lesssim |\pa u|\circ r_s$)  gives
\begin{equation}\label{eq_hr}
  ||u||_{L^{\tilde{p}}(N^{s\ep})} \lesssim  s^{-\frac{\nu}{p}}\cdot ||u||_{L^p(N^{s\ep})}^{1-\beta}\cdot  || u||_{W^{1,p}(N^{s\ep})}^{\beta}.
\end{equation}

Set now $\mu:=\frac{\tilde{\mu}}{(\mut+1)(2\nu+\mut+\beta)}$, fix  $q\in (p,p+\mu]$,  and set \begin{equation}\label{eq_theta}\theta:= \frac{\pet(q-p)}{q(\tilde{p}-p)}= \frac{\tilde{p}(q-p)}{q\tilde{\mu}},\end{equation} in order to have $\frac{1}{q}=\frac{\theta}{ \tilde{p}}+\frac{(1-\theta)}{ p}$. Note that as $(q-p)\le \mu$ and $\frac{\tilde{p}}{q}< \mut+1$, it follows from the definitions of $\mu$ and $\theta$ that \begin{equation}\label{eq_maj_theta}\theta < \frac{1}{2\nu+\mut+\beta}. \end{equation}
As $\frac{1}{q}=\frac{\theta}{ \tilde{p}}+\frac{(1-\theta)}{ p}$, by H\"older's inequality (\ref{eq_holder}), we have for all $u\in \cc^\infty(M)$ and all $s\in (0,1)$ 
\begin{equation*}
 ||u||_{L^q(N^{s\ep})} \le ||u||_{L^p(N^{s\ep})} ^{(1-\theta)} \cdot ||u||_{L^{\tilde{p}}(N^{s\ep})} ^{\theta}\; , 
\end{equation*}
so that, substituting (\ref{eq_hr}) in the right-hand-side, we get:
\begin{eqnarray*}
   ||u||_{L^q (N^{s\ep})}& \lesssim&  ||u||_{L^p(N^{s\ep})} ^{(1-\theta)} \cdot \left( s^{-\frac{\nu}{p}} \cdot||u||_{L^p(N^{s\ep})}^{1-\beta}  \cdot || u||_{W^{1,p}(N^{s\ep})}^{\beta}\right)^{\theta}\\
    &=& s^{-\frac{\nu\theta}{p}} \cdot||u||_{L^p(N^{s\ep})}^{(1-\beta\theta)} \cdot || u||_{W^{1,p}(N^{s\ep})}^{\beta\theta}\;.
\end{eqnarray*}

 If $u$ satisfies  $\supp\, u\subset \bou(0_{\R^n},\ep)$, via (\ref{eq_link_forallp}), this entails that if we set  \begin{equation}\label{eq_k_greater}
                                                                               k:=2\nu+\mut+\beta
                                                                              \end{equation}
  then we get (since  $k\in (\beta,\frac{1}{\theta})$, see (\ref{eq_maj_theta}))
\begin{equation}\label{eq_ulq_avant_h2}
    ||u||_{L^q (N^{s\ep})} \lesssim s^{-\frac{\nu\theta}{p}}\cdot \ln \frac{1}{s\ep}\cdot ||u||_{L^p(N^{s\ep})}^{(1-k\theta)} \cdot || u||_{W^{1,p}(N^{s\ep})}^{\beta\theta} \cdot || u||_{W^{1,p}(M^{\ep})}^{(k-\beta)\theta}\;,
\end{equation}
 and consequently:
\begin{equation}\label{eq_es1}
    ||u||_{L^q (M^{\ep})}\overset{(\ref{eq_coarea_sph})}\lesssim  \left( \int_0 ^1 ||u||_{L^q (N^{s\ep})} ^qds \right)^{1/q} \overset{(\ref{eq_ulq_avant_h2})}\lesssim ||g||_{L^q([0,1])} \cdot || u||_{W^{1,p}(M^{\ep})}^{(k-\beta)\theta}\;,
\end{equation}
where  $$g(s):= s^{-\frac{\nu\theta}{p}}\cdot \ln \frac{1}{s\ep}\cdot ||u||_{L^p(N^{s\ep})}^{(1-k\theta)} \cdot || u||_{W^{1,p}(N^{s\ep})}^{\beta\theta} .$$
 To estimate the $L^q$ norm of $g$, we wish to apply (\ref{eq_holder}) with $\kappa=3$ and  $$p_1=p_2:=p, \qquad g_1(s):=||u||_{L^p(N^{s\ep})}, \quad \theta_1:=(1-k\theta),$$ and $$g_2(s):= || u||_{W^{1,p}(N^{s\ep})}, \qquad\theta_2:=\beta\theta,$$ as well as   with  $$g_3(s):=s^{-\frac{\nu\theta}{p} }\cdot \ln\frac{1}{s\ep}, \qquad  \theta_3:= 1,$$ with $p_3$ such that \begin{equation}\label{eq_theta_3} \frac{\theta_1}{p} +\frac{\theta_2}{p} +\frac{1}{p_3}=\frac{1}{q}.\end{equation}
                                                                                                                                                                                                                                                                                        
This requires to check that such $p_3>0$ exists and, for the value of $p_3$ that we will determine, that $g_3$ is $L^{p_3}$, which amounts to check that $\frac{\nu \theta p_3}{p}<1$. 

To see this, observe that as (by (\ref{eq_k_greater}))  $\theta_1+\theta_2=1-(2\nu+\mut)\theta$, we must have (due to (\ref{eq_theta_3}), $\frac{1}{q}=\frac{\theta}{\pet}+\frac{1-\theta}{p}$, (\ref{eq_theta}), and $\pet=p+\mut$) \begin{equation}\label{eq_p_3}
                                                                                      p_3=\frac{\tilde{\mu}pq}{(q-p)(2\nu\pet+\tilde{\mu}(\pet-1))}>0,
                                                                                     \end{equation}
                                                                                     and, by (\ref{eq_theta}), we then also have
                                                      \begin{equation}
                                                            \frac{\nu \theta p_3}{p}=\frac{\nu\pet}{2\nu\pet+\tilde{\mu}(\pet-1)}
                                                            \le\frac{1}{2}                          
                                                       \end{equation}
(since $\pet\ge 1$), yielding the required fact.  By H\"older's inequality (\ref{eq_holder}), we therefore can conclude for $q\in [p,p+\mu]$:
\begin{equation*}
 ||g||_{L^q((0,1))} \lesssim ||u||_{L^p(M^{\ep})}^{(1-k\theta)}\cdot  || u||_{W^{1,p}(M^{\ep})}^{\beta\theta}.
\end{equation*}

Plugging this inequality into (\ref{eq_es1}), we get for such $q$:
\begin{equation*}
 ||u||_{L^q(\mep)}\lesssim ||u||_{L^p(M^{\ep})}^{1-k\theta}\cdot ||u||_{W^{1,p}(M^{\ep})} ^{k\theta},
\end{equation*} which yields (\ref{ineq}), with $\alpha_{p,q}:=k\theta$, so that
 (\ref{eq_alpha_lip}) then immediately follows from (\ref{eq_theta}) and (\ref{eq_k_greater}).  

It remains to check the compactness of the embeddings.  Fix any $p\in [1,\infty)$ and note first that, thanks to (\ref{ineq}), it suffices to check it for $q=p$.

 Let us fix a bounded subset $\F$ of $W^{1,p}(M)$ and note that as the set  $$K_\eta:=\{x\in M :\forall y\in \delta M, \;|x-y|\ge \eta\} $$ is compact (for each $\eta>0$), it can be covered by finitely many coordinate systems of $M$ that are bi-Lipschitz mappings.  It therefore follows from the compactness of classical Sobolev embeddings that every bounded sequence of $W^{1,p}(M)$ has, for every $\eta$, a subsequence whose restrictions to $K_\eta$ constitute a convergent sequence of $L^p(K_\eta)$. We thus only have to check the sometimes called equivanishing condition i.e., that if we set $Z_\eta:=M\setminus K_{\eta}$ then $||u||_{L^p(Z_\eta)}$ tends to zero as $\eta$ goes to zero uniformly in $u\in \F$.  

Note that it follows from (\ref{ineq}) that
\begin{equation}\label{eq_f_bounded_lq}C:=\sup\{||u||_{L^{p+\mu}(M)}:u\in \F\}<\infty.\end{equation}                                                                                                
 Moreover, it easily follows from Cauchy-Crofton's formula that $\hn^m(Z_\eta)$ goes to zero as $\eta$ tends to zero (see \cite[Proposition $5.3.4$]{livre}).
If $b\in [1,\infty)$ is such that $\frac{1}{p+\mu}+\frac{1}{b}=\frac{1}{p}$, by H\"older's inequality, we therefore have:
$$||u||_{L^p (Z_\eta)} \le \hn^m(Z_\eta)^{\frac{1}{b}}\cdot ||u||_{L^{p+\mu}(M)} \overset{(\ref{eq_f_bounded_lq})}{\le}C \hn^m(Z_\eta)^{\frac{1}{b}}, $$
from which we can conclude that $||u||_{L^p (Z_\eta)}$  tends to zero, uniformly in $u\in \F$.
 \end{proof}
\begin{rem}  It seems that the constant $C$ of this theorem could be required to be independent of $p$ and $q$. Moreover,
 alike in the case of Lipschitz manifolds with boundary, it is possible to prove some embedding theorems for $L^p(M)$ into $W^{-1,q}(M)$ for some $q>p$. For technical reasons,  we have decided to study $W^{-1,p}(M)$ in the only case where $M$ is an open (definable) subset of $\R^n$, in which case this kind of theorems is easy to derive from the classical Sobolev's embeddings, as $L^p(M)$ is easy to embed in $L^p(\R^n)$ in this case. 
\end{rem}

\begin{pro}\label{pro_gen_poinc1}(Generalized Poincar\'e inequality)
 Assume that $M$ is connected and let $U$ be a nonempty open subset of $\pad M$. For every $p\in [1,\infty)$, there is a constant $C$ such that we have for each $u\in W^{1,p}(M)$:
 \begin{equation}\label{ineq_poin}
 ||u||_{L^p(M)}\le C ||\pa u||_{L^p(M)}+C||\tra_{\pa M}\, u||_{L^p(U)}.
 \end{equation} 
\end{pro}
\begin{proof}
 We derive it from the compactness of Sobolev's embedding using a very classical argument. Assume that (\ref{ineq_poin}) fails. There is a sequence $u_i\in W^{1,p}(M)$ satisfying $||u_i||_{L^p(M)}\equiv 1$ and $\lim ||\tra_{\pa M} u_i||_{L^p(U)}=\lim ||\pa u_i||_{L^p(M)}=0$. By the compactness of Sobolev's embedding (for $p=q$), extracting a subsequence if necessary,  we can assume this sequence to be convergent in $L^p(M)$. Its limit $u_\infty\in L^p(M)$ then clearly satisfies $\pa u_\infty\equiv 0$, as distribution,  so that, since $M$ is connected, $u_\infty$ must be constant (almost everywhere). Since $||\pa u_i-\pa u_\infty||_{L^p(M)}$ and $|| u_i-u_\infty||_{L^p(M)}$ both tend to zero, we conclude that $u_i$ tends to $u_\infty$ in $W^{1,p}(M)$.  As $||\tra_{\pa M} u_i||_{L^p(U)}$ tends to zero, we deduce that $\tra_{\pa M} u_\infty \equiv 0$ on $U\ne \emptyset$, which means that $u_\infty$, which is constant, must be zero. This contradicts $||u_i||_{L^p(M)}\equiv 1$.
\end{proof}
\begin{rem} This estimate continues to hold for a definable subset  $U\subset \delta M$ of possibly empty interior if $p$ is sufficiently large. Then, of course, one has to write  $||\tra_U u||_{L^p(U,\hn^k)}$, where $k:=\dim U$, in the estimate.\end{rem}

 The combination of (\ref{ineq}) with (\ref{ineq_poin}) gives the famous Gagliardo-Nirenberg-Sobolev inequality:
 \begin{cor} If $M$ is connected then, given a nonempty open subset  $U$ of $\pad M$, for each $p$ there is a constant $C$ such that
 for all  $u\in W^{1,p}(M,U)$:
\begin{equation*}
   ||u||_{L^{p+\mu}(M)}\le C ||\pa u||_{L^p(M)},
\end{equation*}
where $\mu$ is provided by Theorem \ref{thm_embedding}.\end{cor}

In the same spirit, we have the following result, which can be proved with exactly the same proof as Proposition \ref{pro_gen_poinc1}: 
\begin{pro}
  Assume that $M$ is connected and let  $Z\subset M$ be of nonempty interior (in $M$). For every $p\in [1,\infty)$ there is a constant $C$ such that we have for each $u\in W^{1,p}(M)$:
 \begin{equation}\label{ineq_poin2}
  ||u||_{L^p(M)}\le C||u||_{L^p(Z)}+ C ||\pa u||_{L^p(M)}.
 \end{equation}
\end{pro}
This leads us to the following theorem which is another famous feature of Lipschitz domains \cite{denylions, ma} that we extend to the category of bounded subanalytic manifolds. 
\begin{thm}\label{thm_derlp_implique_lp}
 Let $u$ be a distribution on $M$ and let $p\in [1,\infty)$. If $|\pa u|$ is $L^p$ then so is $u$.
\end{thm}
\begin{proof} We can assume that $M$ is a connected set. Note that such a distribution $u$ is an $L^p_{loc}$ function \cite[Ch. 1]{ma} or \cite{denylions}.  Choose a bi-Lipschitz trivial stratification $\Sigma$ of $\mba$ of which $M$ is a stratum (see Remark \ref{rem_pi_S} (\ref{item_whitney})), and denote by $\Sigma_k$ the union of all the strata of dimension less than or equal to $k$, with by convention $\Sigma_{-1}=\emptyset$.

The idea is one more time to argue by induction on the dimension of the strata that touch the domain of the considered function. 
  Set for this purpose, given $A\subset \mba$,
$$\kappa(A):=\max \{k: \Sigma_{k-1}\cap A =\emptyset \}. $$
We shall show the following statements by decreasing induction on $k$:

\noindent $\mathbf{(H_k)}$. Let $A$ be a closed subset of $\mba$ satisfying $\kappa(A)\le k$. 
 If a distribution $u$ on $\oca\cap M$ satisfies  $\pa u \in  L^{p}(\oca\cap M)^n$, where $\oca$ is some open neighborhood of $A$ in $\mba$,  then $u\in L^p(\oca'\cap M)$, for some  open neighborhood $\oca'$ of $A$ in $\mba$.

  If $\kappa(A)=m$ then  $ A$  does not meet $\delta M$,  and, since such a function $u$ is $L^p_{loc}$, it must be $L^p$ on some neighborhood of $A$.
 Let us thus fix $k<m$, and, assuming  $\mathbf{(H_{k+1})}$, take $u$ and $A$ like in the statement of $\mathbf{(H_k)}$.

As $ A$ is compact, it suffices to prove that every point of $A$ has a neighborhood $U$ in $\mba$ such that $u\in L^p (U\cap M)$. 
Take any $x_0\in A$ and let  $S\in \Sigma$ denote the stratum containing this point. Observe that $\dim S\ge k$, and, if $\dim S>k$ then the result comes down from  $\mathbf{(H_{k+1})}$. We thus can assume $\dim S=k$ and will work near $x_0=\orn$.

Let $U_0$, $W_0$, and $\Lambda$ be as defined just before (\ref{eq_FV}) and let,  for $\eta>0$ small, $F^\eta$ and $V^\eta$ be as in (\ref{eq_FV}).  As $\Lambda$ is bi-Lipschitz, we can work as if $u$ were defined on $\Lambda(U_0\cap M)$.

    Let now $\psi:\R \to [0,1]$  be a $\cc^\infty$ function  
such that $\psi\equiv 1$ on $(-\infty,\frac{1}{4})$ and $\psi\equiv 0$ on $[\frac{1}{2},+\infty)$, and set for $(x,y)\in  \bou(0,\ep)\times  W_0$, \begin{equation*}\phi(x,y):=\psi\left(\frac{|x|^2}{\ep^2}\right).\end{equation*}

It suffices to show that both $u\phi$ and $(1-\phi)u$ are $L^p$ in the vicinity of $\orn$.  As $(1-\phi) \equiv 0$ on 
$$ Z:=  \bou(0,\epd)\times  W_0\,,$$
and since $\kappa(\,\adh{A\cap V^\ep\setminus Z}\,)>k$, the assumption $\mathbf{(H_{k+1})}$ entails that $(1-\phi) u$ is $L^p$ on $V^\ep$. 
To establish that so is $\phi u$,
observe that for each $y\in W_0$  the function $$v_y(x):=\phi(x,y) u(x,y), \qquad x\in F^\ep ,$$ satisfies  $|\pa v_y|\in  L^{p}(F^\ep)$ and vanishes for  $|x|$ close to $\ep$. As a matter of fact, by
 (\ref{eq_link_forallpm}),  for each $y\in W_0$, $v_{y}$ is $L^p$ and there is a constant $C$ independent of $y$ such that:
$$||v_{y}||_{L^p(F^\ep)}  \overset{ (\ref{eq_link_forallpm})}\le C ||\pa v_{y}||_{L^p(F^\ep)} \le C||\pa (\phi u)||_{L^p(F_y)}, $$
where $F_y:=F^\ep\times\{y\}$.
Raising it to the power $p$ and integrating with respect to $y$ we get:
$$||u||_{L^p(V^\ep)}\le C||\pa (\phi u)||_{L^p(V^\ep)}\le C|| u\pa \phi||_{L^p(V^\ep)}+C||\phi \pa u||_{L^p(V^\ep)}.$$
By assumption, $\phi |\pa u|$ is $L^p$, and since we have $\pa \phi\equiv 0$ on $Z$ as well as  $\kappa(\adh{A\cap V^\ep\setminus Z})>k$,  the induction hypothesis yields that so is $u\pa_i \phi$ on $V^\ep$, for all $1\le i\le n$. 
\end{proof}

\end{section}

\begin{section}{Morrey's embedding}\label{sect_morrey}
We recall that a function $u$ is said to be  {\bf H\"older continuous (with exponent $\alpha>0$)} on a set $P\subset \R^n$  when there is a constant $C$  such that for all $x$ and $y$ in $P$
\begin{equation}\label{eq_holder_outer} |u(x)-u(y)|\le C |x-y|^\alpha.\end{equation}

            Given $x$ and  $y$ in the same connected component of a definable set $A$,  let:$$d_{A}(x,y):=\inf \{ lg(\gamma):\gamma:[0,1]\to A, \mbox{ $\cc^0$ definable arc joining } x\mbox{ and } y\},$$
  where $lg(\gamma)$ stands for the length of the definable arc $\gamma$ (which is piecewise $\cc^1$).
  It will be convenient to extend this function by setting $d_A(x,y)=1$ if $x$ and $y$ lie in different connected components.  This defines a metric on $A$, generally referred as {\bf the inner metric} of $A$.

A function $u$ on $A$ is {\bf H\"older continuous with respect to the inner metric} (with exponent $\alpha>0$) if there is  a constant $C$  such that for all $x$ and $y$ in $A$
$$ |u(x)-u(y)|\le C\; d_A(x,y)^\alpha.$$
When $\alpha=1$,  $u$ will be said to be {\bf Lipschitz with respect to the inner metric}. 

 To avoid any confusion, a function which is H\"older continuous in the sense of (\ref{eq_holder_outer}) will sometimes be said to be H\"older continuous  with respect {\it to the  outer metric}. We denote by $\cc^{0,\alpha}(A)$ (resp. $\check{\cc}^{0,\alpha}(A)$) the space of the functions on $A$ that are H\"older continuous  with respect to the outer (resp. inner) metric with exponent $\alpha$.  It is worthy of notice that each element of $W^{1,\infty}(M)$ coincides almost everywhere with an element of $\check{\cc}^{0,1}(M)$.

\begin{rem}\label{rem_inout}
Of course, if a function is H\"older continuous (with respect to the outer metric), it is H\"older continuous with respect to the inner metric with the same exponent, i.e., $\cc^{0,\alpha}(M)\subset \check{\cc}^{0,\alpha}(M)$, for all $\alpha>0$.  On the other hand, it easily follows from \L ojasiewicz's inequality and some basic  Lipschitz properties of definable sets, that if  $M$ is normal (and bounded) then   $\check{\cc}^{0,\alpha}(M)\subset \cc^{0,\theta\alpha}(M)$, for some $\theta>0$ depending on $M$.
\end{rem}

Given $u\in \cc^{0,\alpha}(M)$ we set \begin{equation*}
                                      |u|_{\cc^{0,\alpha}(M)}:= \sup \big\{\frac{|u(x)-u(y)|}{|x-y|^\alpha}:x,y\in  M, x\ne y\big\},
                                     \end{equation*}
                                     as well as 
                                     \begin{equation*}
                                      ||u||_{\cc^{0,\alpha}(M)}:=||u||_{L^\infty(M)}+|u|_{\cc^{0,\alpha}(M)}.  
                                     \end{equation*}
It is well-known that $  ||\cdot ||_{\cc^{0,\alpha}(M)}$ is a norm that makes of $\cc^{0,\alpha}(M)$ a Banach space. The norm $  ||\cdot ||_{\check{\cc}^{0,\alpha}(M)}$  is defined analogously.

\begin{thm}\label{thm_holder}If $M$ is normal then
 there  is a positive real number $\alpha$ such that for all $p$  sufficiently large, the embedding $$(\cc^\infty(\mba),||\cdot||_{W^{1,p}(M)} )\hookrightarrow (\cc^{0,\alpha}(M),||\cdot||_{\cc^{0,\alpha}(M)})$$ is continuous and consequently extends to a compact embedding $W^{1,p}(M)\hookrightarrow \cc^{0,\alpha}(M)$.
\end{thm}

\noindent{\it Proof.} We shall show that the embedding is continuous. That it extends to $W^{1,p}(M)$ then follows from Theorem \ref{thm_trace} $(i)$ (for $p$ sufficiently large), and its compactness comes down from Ascoli-Arzela's Theorem.

Let $\Sigma$ be a locally bi-Lipschitz trivial stratification of $\mba$ compatible with $M$ of which $M$ is a stratum (see Remark \ref{rem_pi_S} (\ref{item_whitney})). Given $k \le m$, we denote by $\Sigma_k$ the union of the strata of dimension less than or equal to $k$, and,
given an open subset $U$ of $M$, we let $$\kappa (U):=\min \{ k: \Sigma_k\cap \adh{U} \ne \emptyset\}.$$

We are going to prove the following statements by downward induction on $k\in\{ 0,\dots, m\}$:

\noindent {\bf $(\mathbf{A}_k)$.} For every $U\subset M$ such that $\kappa (U)\ge k$, there are positive real numbers $p_0$ and $\alpha$ such that the linear mapping $\Psi_U:(\cc^\infty(\mba),||\cdot ||_{W^{1,p}(M)})\to \cc^{0,\alpha}(U)$, $u \mapsto u_{|U} $ is continuous for all $p>p_0$.

As $M$ is a Lipschitz manifold with boundary at each point of an $(m-1)$ dimensional stratum,   {\bf $(\mathbf{A}_{m-1})$} comes down from the classical version of Morrey's embedding.  The theorem follows from {\bf $(\mathbf{A}_{0})$}.
Assume the result true for some $(k+1)$, $k\in \{0,\dots, m-2\}$, and let us prove it for $k$. Fix for this purpose an open subset $U$ of $M$ satisfying $\kappa (U)=k$.

To summarize, we have to show, assuming {\bf $(\mathbf{A}_{k+1})$}, that we can find $\alpha>0$ such that we have for all $p$ sufficiently large and $u\in \cc^\infty(\mba)$:
\begin{equation}\label{eq_morrey_claim}
||u||_{\cc^{0,\alpha}(U)}\lesssim ||u||_{W^{1,p}(M)}. 
\end{equation}
Before starting the first step, let us make some observations that are consequences of the induction assumption.
 We already know by induction that inequality (\ref{eq_morrey_claim}) holds if we replace $U$ with any open subset $V$ of $M$ satisfying $\kappa (V)>k$, for each $u\in W^{1,p}(M)$ (assertion $(\mathbf{A}_{k+1})$ yields it for $u\in \cc^\infty(\mba)$ and consequently for any $u\in W^{1,p}(M)$, by density). Since we can use a cutoff function,  given such $V$,  we actually have
 \begin{equation}\label{eq_obs}||u||_{\cc^{0,\alpha}(V)}\lesssim ||u||_{W^{1,p}(V'\cap M)},\end{equation}                                                                                                                                                                                                                                                                                                                                                     for every $u\in W^{1,p}(V'\cap M)$ (for $p$ large enough), if $V'$ is any neighborhood of $\adh{V}$ in $\mba$. Of course, the constant involved in this inequality depends on the chosen neighborhood $V'$.

Since  $U$ is relatively compact and because we can use a partition of unity, it is enough to deal with the case of a function $u$ whose support fits in some prescribed neighborhood  $U_\xo$ of some point $x_0\in \overline{U}$ that we can choose arbitrarily small. Let us thus take any $\xo\in \adh{U}$ and let $S$ denote  the element of $\Sigma$ that contains $x_0$. We may assume $S$ to be of dimension $k$ (since, in the case $\dim S> k$, $x_0$ would have a neighborhood $\uxo$ such that $\kappa(\uxo)>k$, and on the other hand $\xo\in \adh{U}$ entails $\dim S \ge \kappa (U)=k$).


\bigskip
 
  We establish in step \ref{step_0} estimate (\ref{eq_morrey_claim}) in the case $k=0$ (which is the easiest case), $x_0=0_{\R^n}\in \delta M$ and $U=M^{\epd}$,
  the notations  $\ep$, $r_s$, $M^\eta$, and $N^\eta$ being as introduced in section \ref{sect_lcs_mba}.
  Note that, as $r_s$ is bi-Lipschitz for every $s\in (0,1)$, it follows from \L ojasiewicz's inequality that there are positive numbers $\gamma$ and $C$ such that for all $s\in (0,1]$ and all $x,y$ in $\mep$:
\begin{equation}\label{eq_r_lip}
 |x-y|\frac{s^{\gamma}}{C} \le |r_s(x)-r_s(y)| \le \frac{C}{s^\gamma} |x-y|. 
\end{equation}

\begin{step}\label{step_0} Assuming {\bf $(\mathbf{A}_{1})$}, we show that  there are positive numbers $\alpha$, $C$ and $p_0$, depending only on $\ep$ and $M$, such that for all $p\ge p_0$ and all $u\in \cc^\infty(\adh{\mep})$: 
\begin{equation*}
 ||u||_{\cc^{0,\alpha}(M^\frac{\ep}{4})} \le C||u||_{W^{1,p}(\mep)}.
\end{equation*}\end{step}

\begin{proof}[proof of step \ref{step_0}.] By Theorem \ref{thm_trace}, we have $|u(0)|\lesssim ||u||_{W^{1,p}(\mep)}$. Hence, since we can
subtract $u(0)$ if necessary, we see that it suffices to prove the claimed estimate for a function $u$ that vanishes at the origin (which will allow us to apply (\ref{eq_u_circ_r_sm}), see (\ref{eq_hr_0u}) below).

 For $\eta\le \ep$, let  $P^\eta:=M^{\eta}\setminus M^\frac{\eta}{2}$, and apply  {\bf $(\mathbf{A}_{1})$} to $P^\epd$. This provides  constants  $\tilde{p}_0$ and $\alphat$  such that for all $v\in W^{1,p}(\mep)$ (see  (\ref{eq_obs})), with $p> \tilde{p}_0$, we have 
\begin{equation}\label{eq_hr_0}
 ||v||_{\cc^{0,\alphat}(P^\epd)} \lesssim ||v||_{W^{1,p}(\mep)}.
\end{equation}
As $\mep$ is bounded, this inequality is still valid if we change $\alphat$ for a smaller value (for a possibly bigger constant). In particular, we may assume $\alphat<\frac{1}{2\gamma}\,$, where $\gamma$ is the exponent that appears in (\ref{eq_r_lip}).

Applying (\ref{eq_hr_0}) to $v=u\circ r_s$, $s\in (0,1]$, $u\in \cc^\infty(\adh{\mep})$ satisfying $u(0)=0$, $p$ large enough, by (\ref{eq_u_circ_r_sm}),   we immediately see that
\begin{equation}\label{eq_hr_0u}
  ||u\circ r_s||_{\cc^{0,\alphat}(P^\epd)} \lesssim s^{1-\frac{\nu}{p}}||u||_{W^{1,p}(M^{s\ep})},
\end{equation}
 which implies
 \begin{equation}\label{eq_linfty_0} 
                                                                                    ||u||_{L^\infty(P^{s\epd})}=||u\circ r_s||_{L^\infty(P^\epd)}\lesssim s^{1-\frac{\nu}{p}}   ||u||_{W^{1,p}(M^{s\ep})}.
                                                                                  \end{equation}
As the constant involved in this estimate is independent of $s$, we already see that for $p$ large enough, we have for $s\in (0,1]$ and $u\in \cc^\infty(\adh{\mep})$  vanishing at the origin \begin{equation}\label{eq_linfty_0_M}                                                                                                                                                                          ||u||_{L^\infty(M^{s\epd})}\lesssim s^{1-\frac{\nu}{p}}  ||u||_{W^{1,p}(\mep)}.
                            \end{equation}
        It therefore only remains to show that there are positive constants $C$, $p_0$, and $\alpha$ such that for all $p\ge p_0$, $u\in \cc^\infty(\adh{\mep})$ vanishing at the origin,  as well as  $x$ and $y$ in $M^\frac{\ep}{4}$, we have:
\begin{equation}\label{eq_claim_step1}
 |u(x)-u(y)|\le C  ||u||_{W^{1,p}(\mep)}\cdot |x-y|^\alpha.
\end{equation}
To see this, observe that (\ref{eq_hr_0u}) also implies that  for such $u$ we have for all $x$ and $y$ in $P^\epd$:
\begin{equation}
 |u(r_s(x))-u(r_s(y))|\lesssim s^{1-\frac{\nu}{p}} ||u||_{W^{1,p}(\mep)}\cdot |x-y|^\alphat,
\end{equation}
which, thanks  to (\ref{eq_r_lip}),  entails
\begin{equation}
 |u(r_s(x))-u(r_s(y))|\lesssim s^{1-\frac{\nu}{p}} ||u||_{W^{1,p}(\mep)}\cdot |r_s(x)-r_s(y)|^\alphat\cdot s^{-\alphat\gamma}.
\end{equation}
As $\alphat\gamma<\frac{1}{2}$, we see that if $p$ is sufficiently large then we have for $s\in (0,1]$ and $u\in \cc^\infty(\adh{\mep})$ vanishing at the origin: \begin{equation}\label{eq_psep}
 |u|_{\cc^{0,\alphat}(P^{s\epd})}\lesssim  ||u||_{W^{1,p}(\mep)}.
\end{equation}

To show that this implies (\ref{eq_claim_step1}), take $x$ and $y$ in $M^\frac{\ep}{4}$, and set $s:=\frac{3|x|}{\ep}$, in order to have $x\in P^{s\epd}$.  It is no loss of generality to assume $|x|\ge |y|$. If $y \in P^{s\epd}$, by (\ref{eq_psep}), the desired inequality holds (with $\alpha=\alphat$ in this case) for the couple $(x,y)$. Otherwise, $y\in M^{\frac{s\ep}{4}}=M^{\frac{3|x|}{4}}$, which entails  $|x-y|>\frac{|x|}{4}=\frac{s\ep}{12}$, which implies in turn for $p$ sufficiently large
\begin{equation*}
|u(x)-u(y)|\le |u(x)|+|u(y)|\overset{(\ref{eq_linfty_0_M})}{\lesssim}  s^{\frac{1}{2}}\cdot ||u||_{W^{1,p}(\mep)}  \lesssim |x-y|^{\frac{1}{2}}\cdot ||u||_{W^{1,p}(\mep)},
\end{equation*}
yielding the claimed fact (here, we have applied  (\ref{eq_linfty_0_M}) twice, once to get a bound for $|u(x)|\le  ||u||_{L^\infty(M^{s\epd})}$ and once to get a bound for  $|u(y)|\le  ||u||_{L^\infty(M^{s\epd})}$), completing the proof of step $1$.
\end{proof}

\bigskip

By definition of bi-Lipschitz triviality, the point $\xo$ that we have fixed has a neighborhood $U_\xo$ in $\mba$, for which there is a bi-Lipschitz homeomorphism $\Lambda: U_\xo \to (\pi^{-1}(0_{\R^n})\cap \adh{M})\times W_\xo$, where $\pi:U_\xo \to S$ is a $\cc^\infty$ definable (see Remark \ref{rem_pi_S} (\ref{item_retractions}))  retraction and $W_\xo$ is a neighborhood of $x_0$ in $S$.

 The next following steps are devoted to the proof of the induction step in the case $k>0$.      
      The idea is to split the proof of (\ref{eq_morrey_claim}) for $k>0$ into two steps: we first establish the H\"older condition for the restriction of $u$ to the stratum $S$ (step \ref{step_stra} below) and then prove it for a couple of points $(x,y)$ for which $\pi(x)$ and $\pi(y)$ are close to each other (step \ref{step_v}). Step \ref{step_ko} will then derive (\ref{eq_morrey_claim}) from these two steps.  
We therefore from now assume $k>0$.  Before proceeding, we introduce some notations.

We can assume that $F:=\pi^{-1}(0_{\R^n}) \cap M $ is a $\cc^\infty$ manifold (see Remark \ref{rem_pi_S} (\ref{item_whitney})).
  Set for simplicity for $\eta\le \ep$:
\begin{equation*}
F^\eta:=\pi^{-1}(0_{\R^n}) \cap M^\eta  \et U_0^\eta:=F^\eta \times \bou(0_{\R^k},\eta). 
\end{equation*}
Since  we can work up to a bi-Lipschitz homeomorphism, we will identify $U_\xo$ with $U_0^\ep$, $\xo$ with the origin, and $S\cap U_\xo$ with $W_0:=\{0_{\R^{n-k}}\}\times \bou(0_{\R^k},\ep)$.

 Observe that $\pi$ is then induced by the canonical projection onto $\{0_{\R^{n-k}}\}\times \R^k$. 
We shall  sometimes regard $W_0$ as a subset of $\R^k$, identifying $\bou(0_{\R^k},\ep)$ with $\{0_{\R^{n-k}}\}\times \bou(0_{\R^k},\ep)$ (and $x\in\bou(0_{\R^k},\ep) $ with $(0_{\R^{n-k}},x)\in W_0$). 

Apply now Theorem \ref{thm_local_conic_structure} to the germ of $F$ at the origin. This provides a Lipschitz definable mapping $r:[0,1]\times F^\ep\to F^\ep$ such that for every $s\in (0,1]$, $r_s(F^\ep)=F^{s\ep}$,
%
 that enables us define a retraction by deformation of $U_0^\ep$ onto the origin: \begin{equation}\label{eq_tir}\tir:[0,1]\times U_0^\ep\to U_0^\ep, \qquad (s,z)\mapsto \tir_{s}(z):= \big(r_s(z-\pi(z)),s\pi(z)\big).\end{equation}
                                                                                                                                                                                                                                                                                                                                                                                                            
 More generally, given $x$ in the set \begin{equation}\label{eq_wprime_0}
                                W_0':=\{0_{\R^{n-k}}\}\times \bou(0_{\R^k},\epq) 
                              \end{equation}
and $z\in U_0^\ep$ such that $(z-x)\in U_0^\ep$, we set \begin{eqnarray*}
                                             \tir_{x,s}(z):=\tir_s(z-x)+x,
                                            \end{eqnarray*}
in order to get a retraction by deformation onto $\{x\}$.
 
Observe that the second component of this mapping is just a homothetic transformation of $\R^k$ of ratio $s$ centered at $x$. Hence, by (\ref{eq_der_r_s}), we see that there is a constant $C$ such that 
for all $s\in (0,1)$ we have for almost all  $z\in \bou(x,\epd) \cap U_0^\ep$:  \begin{equation}\label{eq_der_r_s_tilda}
       \left|\frac{\pa \tir_x}{\pa s}(s,z)\right|\le C|z-x|.
      \end{equation}
      As $\tir_{x,s}$ is merely $\tir_{s}$ up to a translation, we can require the constant involved in the above estimate to be independent of $x\in W_0'$.  Moreover,  thanks to \L ojasiewicz's inequality, there are positive real numbers $c$ and $\nu$ which are independent of $(x,s)\in W_0'\times (0,1)$ and such that on $\bou(x,\epd) \cap U_0^\ep$ we have:
      \begin{equation}\label{eq_jac_tir}
       \jac \tir_{x,s} \ge c s^\nu.
      \end{equation}


\begin{step}\label{step_stra}  There exist positive real numbers $C$ and $\mu$, depending only on $U_0^\ep$, such that for all $p$ sufficiently large,  we have for all\footnote{We only have to show (\ref{eq_morrey_claim}) for a function $u$ which is smooth up to the boundary but since we argue up to a bi-Lipschitz trivialization of the stratification we will establish it for $u\in \cc^{0,1}(\adh {U_0^\ep})$.} $u\in \cc^{0,1}(\adh {U_0^\ep})$: \begin{equation}\label{eq_step_S}
                |u|_{\cc^{0,1-\frac{\mu}{p}} (W'_0)}\le C || u||_{W^{1,p}(U_0^\ep)}.
                       \end{equation}\end{step}

                 \begin{proof}[proof of step \ref{step_stra}.]   Since strata are smooth, we will here follow an argument used to prove Morrey's inequality on smoothly bounded domains of $\R^n$.    For $x$ and $y$ in $W'_0=\{0_{\R^{n-k}}\}\times \bou(0_{\R^k},\epq)$,  let   $$A_{x,y}:=M\cap \bou(x,\eta)\cap \bou(y,\eta), \;\; \mbox{where }\;\;\; \eta:=|x-y|.$$ 
We first prove  that there are positive real numbers $C$ and $\mu$ such that for all $x$ and $y$ in $W'_0$ we have: \begin{equation}\label{eq_vol_A_x}
         |x-y|^\mu \le C\, \hn^m(A_{x,y}).
        \end{equation}
        To prove this fact, observe that as obviously $$F^\frac{\eta}{4} \times \left(W_0\cap \bou \big(\frac{x+y}{2},\frac{\eta}{4}\big)\right) \subset A_{x,y}$$  (where as above $\eta=|x-y|$), there must be $C>0$  such that for all $x$ and $y$ in $W_0'$ we have \begin{equation}\label{eq_vol_axp}
              \eta^k\cdot \hn^{m-k}(F^\frac{\eta}{4})        \le  C\, \hn^m(A_{x,y}) .
                                  \end{equation}
 It is well-known that (see \cite{lr} or \cite[section $5.2$]{livre}), as $F$ is definable, there is a constant (also denoted $C$) and a real number $l$ such that $\zeta^{l}\le C\hn^{m-k}(F^\zeta) $, for all $\zeta>0$, which means that  (\ref{eq_vol_axp}) yields (\ref{eq_vol_A_x}).
 
 To show (\ref{eq_step_S}), observe now that if $x$ and $y$ belong to $W_0'$, then, by the fundamental theorem of calculus (it is enough to establish (\ref{eq_step_S}) for a function $u\in \cc^\infty(\adh{U_0^\ep})$), we have:
\begin{eqnarray}\label{eq_uxz}\nonumber
 \int_{A_{x,y}}|u(z)-u(x)|\,dz&\le&  \int_{A_{x,y}}\int_0 ^1  \left| \pa u(\tir_{x,s}(z))\cdot \frac{\pa \tir_{x}}{\pa s}(s,z)\right| ds\, dz\\ \nonumber& \overset{(\ref{eq_der_r_s_tilda})}{ \lesssim} &\int_{A_{x,y}}\int_0 ^1 |\pa u(\tir_{x,s}(z))|\cdot |z-x|\, ds\, dz\\ &\le& |x-y|\int_{A_{x,y}}\int_0^1|\pa u(\tir_{x,s}(z))|\,ds\, dz,\end{eqnarray}
 since $|z-x|<|x-y|$, for all $z\in A_{x,y}$. Notice now that if we write
$$ |u(x)-u(y)|\le |u(z)-u(x)|+|u(z)-u(y)|, $$
and then integrate this inequality with respect to $z\in A_{x,y}$, we get
\begin{eqnarray}\label{eq_diff}
\nonumber&& \hn^m(A_{x,y})\,|u(x)-u(y)|= \int_{A_{x,y}}|u(z)-u(x)|dz\; +\;\int_{A_{x,y}}|u(z)-u(y)|dz \\
&&\overset{(\ref{eq_uxz})}{\lesssim} |x-y| \int_{A_{x,y}}\int_0 ^1|\pa u(\tir_{x,s}(z))|dsdz+|x-y| \int_{A_{x,y}}\int_0 ^1|\pa u(\tir_{y,s}(z))|dsdz.\end{eqnarray}
We now can estimate the integrals that appear in (\ref{eq_diff}) as follows:                                                                                                                               
\begin{eqnarray*}
  \int_{A_{x,y}}\int_0 ^1|\pa u(\tir_{x,s}(z))|ds\,dz\overset{(\ref{eq_jac_tir})}{\lesssim} \int_{A_{x,y}}\int_0 ^1|\pa u(\tir_{x,s}(z))|\cdot \jac \tir_{x,s} (z)^{1/p}\cdot  s^{-\nu/p}ds\,dz\end{eqnarray*}
 \begin{eqnarray*}
 &\le& \left( \int_{A_{x,y}}\int_0 ^1|\pa u(\tir_{x,s}(z))|^p\cdot \jac \tir_{x,s} (z)ds\,dz\right)^{1/p}  \cdot     \left(  \int_{A_{x,y}}\int_0 ^1  s^{-\nu p'/p}ds\, dz\right)^{1/p'}\\
 &=&\left(\int_0 ^1 \int_{\tir_{x,s}(A_{x,y})}|\pa u(z)|^p dz \,ds\right)^{1/p}  \cdot     \left(\int_{A_{x,y}} \int_0 ^1   s^{-\nu /(p-1)}ds dz\right)^{1-1/p}\\
  &\le&\left(\int_0 ^1||\pa u||_{L^p(U_0^\ep)}^p ds\right)^{1/p}  \cdot     \left(\int_{A_{x,y}} \int_0 ^1   s^{-\nu /(p-1)}ds\,dz\right)^{1-1/p}\\
    &\lesssim& ||\pa u||_{L^p(U_0^\ep)}  \cdot     \hn^m(A_{x,y})^{1-1/p}\quad \mbox{(for $p\ge\nu+2$).}
\end{eqnarray*}
By the symmetry of our assumptions, this computation is valid to estimate the integral of $ |\pa u(\tir_{y,s}(z))|$ as well,  so that, plugging the two  inequalities resulting from this computation  into (\ref{eq_diff}), we get   
\begin{eqnarray*}
 |u(x)-u(y)|\lesssim |x-y|  \cdot    \hn^m(A_{x,y})^{-1/p} \cdot||\pa u||_{L^p(U_0^\ep)} \overset{(\ref{eq_vol_A_x})}{\lesssim}|x-y|^{1-\frac{\mu}{p}} \cdot ||\pa u||_{L^p(U_0^\ep)} ,
 \end{eqnarray*}
 yielding the claimed estimate and completing step \ref{step_stra}.\end{proof}

 We are still using the notations introduced before step \ref{step_stra}.   
   Step \ref{step_v} will provide an estimate of the norm of the restriction of a function  $u\in \cc^{0,1}(\adh{U_0^{\ep}})$ to a neighborhood of a fiber $\pi^{-1}(a)$, $a\in W'_0$  (with a constant independent of $a$),  proceeding in a similar way as in the case $k=0$ (step \ref{step_0}).  
    Its proof will require a generalization of (\ref{eq_u_circ_r_sm}) to $\tir_s$ (see (\ref{eq_tir}) for $\tir_s$) that we present now (inequality (\ref{eq_utir_wp}) below). 
    
    For $u\in L^p(U_0^\epd)$ and $s\in (0,1]$, we have:
\begin{equation*}\label{eq_u_tir}  ||u\circ \tir_s||_{L^p(U_0^\frac{\ep}{4})}\overset{(\ref{eq_jac_tir})}\lesssim s^{-\frac{\nu}{p}}\cdot||u||_{L^p(\tir_s(U_0^\frac{\ep}{4}))}\le s^{-\frac{\nu}{p}}\cdot ||u||_{L^p(M^{s\epd})} ,\end{equation*}
    since $U_0^\frac{\ep}{4}\subset M^\epd$.  In particular, in the case where $u$ belongs to $\cc^{0,1}\big(\;\adh{U_0^{\ep/2}}\;\big)$ and satisfies $u(0)=0$, by (\ref{theta_contm}), we get for $s\in (0,1]$ (for $p$ large):
    \begin{equation}\label{eq_utir_lp}
      ||u\circ \tir_s||_{L^p(U_0^\frac{\ep}{4})}\lesssim s^{1-\frac{\nu}{p}} \cdot ||u||_{W^{1,p}(M^{\epd})} .
    \end{equation}
Furthermore, for such $u$ and $s$
$$||\pa (u\circ \tir_s)||_{L^p(U_0^\frac{\ep}{4})}= || ^{\mathbf{t}}D\tir_s\left(\pa u\circ \tir_s\right)||_{L^p(U_0^\frac{\ep}{4})}\lesssim s  || \pa u\circ \tir_s||_{L^p(U_0^\frac{\ep}{4})}\overset{(\ref{eq_jac_tir})}\lesssim s^{1-\frac{\nu}{p}}  || \pa u||_{L^p(M^{\epd})}$$
(again, since $U_0^\frac{\ep}{4}\subset M^\epd$).  Combining this estimate with (\ref{eq_utir_lp}), we conclude that for $u\in \cc^{0,1}(\;\adh{U_0^{\ep/2}}\;)$ vanishing at the origin, we have:
 \begin{equation}\label{eq_utir_wp}
      ||u\circ \tir_s||_{W^{1,p}(U_0^\frac{\ep}{4})}\lesssim s^{1-\frac{\nu}{p}} \cdot ||u||_{W^{1,p}(M^{\epd})} .
    \end{equation}
    
          The statement of the next step requires to introduce the following open neighborhoods of $\pi^{-1}(a)\cap M^\eta$ in $U_0^\ep$, for  $a\in \{0_{\R^{n-k}}\}\times \bou(0_{\R^k},\frac{\ep}{4})$ and  $\eta\in (0,\ep)$: $$V^\eta_a:= \{ x\in U_0^\ep: 2|\pi(x)-a| < |x-\pi(x)| < \eta \}.$$

 \begin{step}\label{step_v}There are positive constants $C$, $p_0$, and $\kappa$ such that for all $p>p_0$, $u\in \cc^{0,1}(\adh{U_0^{\ep}})$, and all $a\in \{0_{\R^{n-k}}\}\times \bou(0_{\R^k},\frac{\ep}{4})$, we have \begin{equation}\label{eq_step_cone}                                                                                                                                                                                  
                   ||u||_{\cc^{0,\kappa}(V_a^\frac{\ep}{16})} \le C ||u||_{W^{1,p}(U_0^\ep)}.
                   \end{equation}\end{step}

   \begin{proof}[proof of step \ref{step_v}.] We first focus on the case $a=0_{\R^n}$.
    For $\eta\in (0,\ep)$, set $Z^\eta:=V^\eta _0\setminus\adh{ V_0^{\eta/2}}$ and   apply $(\mathbf{A}_{k+1})$ to $Z^{\frac{\ep}{8}}$. This provides positive constants $C$, $p_0$, and $\kappa$ such that for all $p>p_0$, we have for all $v\in \cc^{0,1}(\adh{U^{\ep}_0})$ (see (\ref{eq_obs})):
     \begin{equation}\label{eq_hrk}
      ||v||_{\cc^{0,\kappa}(Z^{\frac{\ep}{8}})} \le C ||v||_{W^{1,p}(U_0^\frac{\ep}{4})}.
     \end{equation}
     Note that we can change $\kappa$ for a smaller constant. In particular, we can assume that $\kappa<\frac{1}{2\gamma}$, where $\gamma$ is the exponent for which $\tir_s$ satisfies (\ref{eq_r_lip}). 

     Given any function $u\in \cc^{0,1}(\adh{U_0^\ep})$, by Theorem \ref{thm_trace}, we have
 $|u(0)|\lesssim ||u||_{W^{1,p}(U_0^\frac{\ep}{4})}.$
    Consequently, since it is possible to subtract $u(0)$ if necessary, we see that it suffices to establish the desired estimate for a function which  vanishes at the origin.

    Notice first that we have for  $p$ sufficiently large,  $s\in (0,1]$, and  $u\in \cc^{0,1}(\adh{U_0^\ep})$ vanishing at the origin:
    $$ ||u||_{L^\infty(Z^{\frac{s\ep}{8}}_0)}= ||u\circ \tir_s||_{L^\infty(Z^{\frac{\ep}{8}}_0)} \overset{(\ref{eq_hrk})}{\lesssim}  ||u\circ \tir_s||_{W^{1,p}(U_0^\frac{\ep}{4})}\overset{(\ref{eq_utir_wp})}{\lesssim} s^{1-\frac{\nu}{p}}\cdot||u||_{W^{1,p}(M^\epd)} .$$
    As the involved constants are independent of $s$ and since $M^\epd \subset U_0^\epd$,  this implies that for such $u$ and $s\in (0,1]$ we have for $p$ large enough:
     \begin{eqnarray}\label{eqn_linfty_u_Vs} ||u||_{L^\infty(V^{\frac{s\ep}{8}}_0)}\lesssim s^{\frac{1}{2}}||u||_{W^{1,p}(U_0^\epd)}   .\end{eqnarray}
 We are going to show an analogous estimate for $|u|_{\cc^{0,\alpha}(V^{\frac{\ep}{16}}_0)}$ (see (\ref{eq_similar_for_vo}) below).
 Note first that (\ref{eq_hrk}) also implies that for  $u\in \cc^{0,1}(\adh{U_0^\ep})$ we have for all $x$ and $y$ in $Z_0^\frac{\ep}{8}$ and all $s\in (0,1]$:
\begin{equation*}
 |u(\tir_s(x))-u(\tir_s(y))|\lesssim  ||u\circ \tir_s||_{W^{1,p}(U_0^\frac{\ep}{4})}\cdot |x-y|^\kappa,
\end{equation*}
which, via (\ref{eq_utir_wp}) and (\ref{eq_r_lip}) (for $\tir_s$), entails  in the case where $u$ vanishes at the origin
\begin{equation}
 |u(\tir_s(x))-u(\tir_s(y))|\lesssim s^{1-\frac{\nu}{p}} ||u||_{W^{1,p}(U_0^\epd)}\cdot |\tir_s(x)-\tir_s(y)|^\kappa\cdot s^{-\kappa\gamma}.
\end{equation}
As $\kappa\gamma<\frac{1}{2}$ and since  $\tir_s(V_0^\frac{\ep}{8})=V_0^\frac{s\ep}{8}$, we have shown that there are constants $p_0$ and $C$ such that if $p>p_0$  then we have for all $u\in \cc^{0,1}(\adh{U_0^\ep})$ vanishing at the origin, all $s$, and all $x$ and $y$ in $Z_0^\frac{s\ep}{8}$:
 \begin{equation}\label{eq_similar_for_zo}
   |u(x)-u(y)|\le C ||u||_{W^{1,p}(U_0^\epd)} |x-y|^\kappa.
 \end{equation}
  We claim that this estimate continues to hold if we replace $Z_0^\frac{s\ep}{8}$ with $V_0^\frac{\ep}{16}$, i.e.,  that for $u\in \cc^{0,1}(\adh{U_0^\ep})$ vanishing at the origin:
     \begin{equation}\label{eq_similar_for_vo}
   |u|_{\cc^{0,\kappa}(V_0^\frac{\ep}{16})}\le C ||u||_{W^{1,p}(U_0^\epd)},
 \end{equation}
 for some possibly bigger constant $C$.
 
 To show this take any  $x$ and $y$ in $V_0^\frac{\ep}{16}$, with $|y-\pi(y)|\le |x-\pi(x)|$, and take any $s\in (\frac{8|x-\pi(x)|}{\ep},\frac{10|x-\pi(x)|}{\ep})$ (so that we have $x\in Z_0^{s\frac{\ep}{8}}$).  If $y$ also belongs to $Z_0^{s\frac{\ep}{8}}$ then the claimed inequality follows from (\ref{eq_similar_for_zo}). Otherwise, we must have $|y-\pi(y)|<\frac{s\ep}{16}$,  which  implies that ($z\mapsto (z-\pi(z))$ being an orthogonal projection, it is $1$-Lipschitz) $$|x-y|\ge |x-\pi(x)|-|y-\pi(y)| \ge \frac{s\ep}{10}-\frac{s\ep}{16}\ge \frac{s\ep}{32},$$
from which we can conclude that (for $p$ large and $u$ as above)
 \begin{equation*}
|u(x)-u(y)|\le |u(x)|+|u(y)|\overset{(\ref{eqn_linfty_u_Vs})}{\lesssim}  s^{\frac{1}{2}}\cdot ||u||_{W^{1,p}(U_0^\epd)}  \lesssim |x-y|^{\frac{1}{2}}\cdot ||u||_{W^{1,p}(U_0^\epd)},
\end{equation*}
    which yields (\ref{eq_similar_for_vo})  (here, we have applied (\ref{eqn_linfty_u_Vs}) twice to get a bound for both $|u(x)|\le ||u||_{L^\infty(V^{\frac{s\ep}{8}}_0)}$ and $|u(y)|\le ||u||_{L^\infty(V^{\frac{s\ep}{8}}_0)}$).
    
 To prove (\ref{eq_step_cone}), given $u\in W^{1,p}(U_0^\ep)$, $p>p_0$, and any $a\in \{0_{\R^{n-k}}\}\times \bou(0_{\R^k},\frac{\ep}{4})$, it now suffices to apply (\ref{eq_similar_for_vo}) to $\tilde{u}_a(x):=u(x+a)$. As $$||\tilde{u}_a||_{W^{1,p}(U_0^\epd)}\le ||u||_{W^{1,p}(U_0^\ep)},$$  and since $V^{\frac{\ep}{16}}_a$ coincides with $V^{\frac{\ep}{16}}_0$ up to a translation, we get the desired result.
\end{proof}


We are now ready to carry out the induction step in the case $k>0$, which will complete the proof of the theorem:
\begin{step}\label{step_ko}
 There are positive constants $p_0$, $C$, and $\alpha$ such that for all $p>p_0$ and all $u\in \cc^{0,1}(\adh{U_0^\ep})$ we have 
\begin{equation}\label{eq_induc_step_kpo}
 ||u||_{\cc^{0,\alpha}(U_0^\frac{\ep}{16})}\le C||u||_{W^{1,p}(U_0^\ep)},
\end{equation}with $\alpha=\min (\kappa, 1-\frac{\mu}{p})$, where $\mu$ is provided by step \ref{step_stra} and $\kappa$ by step \ref{step_v}.\end{step}
\begin{proof}[proof of step \ref{step_ko}.]
  As the mapping $u\mapsto u(0)$ is Lipschitz on  $W^{1,p}(U_0^\ep)$ (for $p$ large, by Theorem \ref{thm_trace}), it suffices to prove this estimate for the semi-norm $ |u|_{\cc^{0,\alpha}(U_0^\frac{\ep}{16})}$.

 Take $x$ and $y$ in $U_0^\frac{\ep}{16}$. If $|x-y|<\frac{|x-\pi(x)|}{2}$ then $x\in V_a ^\frac{\ep}{16}$, with $a:=\pi(y)$, and the needed inequality then directly follows from (\ref{eq_step_cone}). Otherwise, we have  $|x-y|\ge \frac{|x-\pi(x)|}{2}$, and consequently $|y-\pi(y)|\le 4|x-y| $, so that writing
$$|u(x)-u(y)|\le |u(x)-u(\pi(x))|+  |u(\pi(x))-u(\pi(y))|+ |u(y)-u(\pi(y))|,$$
we see that the desired inequality comes down from (\ref{eq_step_cone}) (applied to both couples $(x,\pi(x))$ and  $(y,\pi(y))$) and (\ref{eq_step_S}) (applied to the couple $(\pi(x),\pi(y))$).
\end{proof}

In the case where $M$ is not assumed to be normal, we have:
\begin{cor}\label{cor_morrey}
There  is a positive real number $\alpha$ such that for all $p$  sufficiently large the embedding $$(W^{1,\infty}(M),||\cdot||_{W^{1,p}(M)} )\hookrightarrow (\check{\cc}^{0,\alpha}(M),||\cdot||_{\check{\cc}^{0,\alpha}(M)})$$  is continuous and therefore extends to a compact embedding $W^{1,p}(M)\hookrightarrow \check{\cc}^{0,\alpha}(M)$.
\end{cor}
\begin{proof}
  It suffices to apply Theorem \ref{thm_holder} to a normalization of $M$.  As $\cc^\infty$ normalizations, as well as their inverse mappings, have bounded first derivative, they are bi-Lipschitz with respect to the inner metric.
\end{proof}

For $u\in W^{1,p}(M)$, we can regard the components of $\pa u$ in the canonical basis of $\R^n$ as the partial derivatives of $u$ (although $M$ may have positive codimension in $\R^n$). We then define the second order partial derivatives as the partial derivatives of these first order partial derivatives, and so on.  We have:

 \begin{cor}\label{cor_infty}
 Let $u\in L^1(M)$. If all the successive partial derivatives of $u$ are $L^1$  then $u$ is Lipschitz with respect to the inner metric. Moreover, in the case where $M$ is normal, such a function $u$ is H\"older continuous.
 \end{cor}
\begin{proof}
Observe first that by Theorem \ref{thm_embedding} (applied with $p=1$), such a function $u$, as well as all its subsequent partial derivatives, belong to $ L^{1+\mu}(M)$. Iterating $k$ times the argument, we thus can show that these functions actually belong to $L^{1+k\mu}(M)$ for all $k\in \N$. By Corollary \ref{cor_morrey}, we deduce that $u$ and $\pa_1 u,\dots, \pa_n u$ all belong to $\check{\cc}^{0,\alpha}(M)$ for some $\alpha>0$, which implies that these are bounded functions. Hence, $u\in W^{1,\infty}(M)$, which entails that it coincides almost everywhere with a function that is Lipschitz with respect to the inner metric.  The last sentence is a consequence of Theorem \ref{thm_holder} (see also Remark \ref{rem_inout}).
\end{proof}
If we define $W^{\infty,1}(M)$ as the space of $L^1$ functions that have all their subsequent partial derivatives $L^1$, the just above corollary exactly shows that  $$W^{\infty,1}(M)\subset W^{1,\infty}(M).$$   
We have actually shown that there is an integer $k$, depending on the Lipschitz geometry of $M$, such that $W^{k,p}(M)\subset W^{1,\infty}(M)$.
\begin{exa}\label{exa_cor_infty}This corollary is no longer true if the underlying manifold is simply required to be definable on a non necessarily polynomially bounded o-minimal structure: if $M$ is the manifold constituted by the interior of the set $Y$ introduced in Example \ref{exa_lcs} and $u(x,y)=\sqrt{x}$ then $u$ as well as all its subsequent partial derivatives belong to $L^1(M)$ but $u$ is not Lipschitz with respect to the inner metric.   \end{exa}

On behalf of all authors, the corresponding author states that there is no conflict of interest.
\end{section}

\end{document}